\documentclass[11pt,oneside, reqno]{amsart}
\usepackage[top=25mm, bottom=25mm, left=25mm, right=25mm]{geometry}
\usepackage{amsfonts,amssymb,amsmath,amsthm,amsxtra}
\usepackage[T1]{fontenc}
\usepackage[utf8]{inputenc}
\usepackage{bm}
\usepackage{mathtools}
\usepackage{dsfont}
\usepackage{bbm}
\usepackage{mathrsfs}
\usepackage{enumitem}

\usepackage{graphics}

\numberwithin{equation}{section}
\newtheorem{theorem}[equation]{Theorem}
\newtheorem{corollary}[equation]{Corollary}
\newtheorem{lemma}[equation]{Lemma}
\newtheorem{proposition}[equation]{Proposition}
\theoremstyle{definition}
\newtheorem{remark}[equation]{Remark}

\newcommand{\CC}{\mathbb{C}}
\newcommand{\DD}{\mathbb{D}}
\newcommand{\EE}{\mathbb{E}}

\newcommand{\GG}{\mathbb{G}}

\newcommand{\II}{\mathbb{I}}
\newcommand{\JJ}{\mathbb{J}}

\newcommand{\LL}{\mathbb{L}}

\newcommand{\NN}{\mathbb{N}}

\newcommand{\QQ}{\mathbb{Q}}
\newcommand{\RR}{\mathbb{R}}

\newcommand{\TT}{\mathbb{T}}

\newcommand{\XX}{\mathbb{X}}

\newcommand{\ZZ}{\mathbb{Z}}

\newcommand{\I}{\mathbb{I}}
\newcommand{\C}{\mathbb{C}}
\newcommand{\N}{\mathbb{N}}

\newcommand{\calF}{\mathcal{F}}
\newcommand{\calB}{\mathcal{B}}

\newcommand{\calO}{\mathcal{O}}

\newcommand{\calN}{\mathcal{N}}

\newcommand{\dif}{\mathrm{d}}
\newcommand{\ex}{\bm{e}}

\newcommand{\ind}[1]{\mathds{1}_{{#1}}}

\newcommand*{\DMO}[1]{\expandafter\DeclareMathOperator\csname #1\endcsname {#1}}
\DMO{ord}
\DMO{supp}
\DMO{Sym}
\DMO{cl}
\DMO{lcm}
\DMO{dist}
\DMO{tr}
\DMO{diam}

\DeclarePairedDelimiter\abs{\lvert}{\rvert}

\DeclarePairedDelimiter\norm{\lVert}{\rVert}
\DeclarePairedDelimiter\card{\lvert}{\rvert}

\DeclarePairedDelimiter\ceil{\lceil}{\rceil}
\DeclarePairedDelimiterX\spr[2]{\langle}{\rangle}{#1,#2}
\newcommand{\ipr}[2]{#1\cdot#2}
\DeclarePairedDelimiterX\Set[2]{\{}{\}}{#1\colon #2}
\DeclarePairedDelimiterX\Seq[1]{(}{)}{#1}

\def\8{\infty}
\DeclareMathOperator{\support}{supp}

\newcommand{\vertiii}[1]{{\left\lvert\kern-0.25ex\left\lvert\kern-0.25ex\left\lvert #1 
    \right\rvert\kern-0.25ex\right\rvert\kern-0.25ex\right\rvert}} 

\begin{document}

\title[Some remarks on
oscillation inequalities]{Some remarks on
oscillation inequalities}

\author{Mariusz Mirek }
\address[Mariusz Mirek]{
Department of Mathematics,
Rutgers University,
Piscataway, NJ 08854-8019, USA \\
\&  School of Mathematics,
  Institute for Advanced Study,
  Princeton, NJ 08540,
  USA
\&
Instytut Matematyczny,
Uniwersytet Wroc{\l}awski,
Plac Grunwaldzki 2/4,
50-384 Wroc{\l}aw
Poland}
\email{mariusz.mirek@rutgers.edu}

\author{Wojciech S{\l}omian}
\address[Wojciech S{\l}omian]{
	  Faculty of Pure and Applied Mathematics, 
	  Wroc\l{}aw University of Science and Technology\\
	  Wyb{.} Wyspia\'nskiego 27,
	  50-370 Wroc\l{}aw, Poland}
    \email{wojciech.slomian@pwr.edu.pl}

\author{Tomasz Z. Szarek}
\address[Tomasz Z. Szarek]{
BCAM - Basque Center for Applied Mathematics,
48009 Bilbao, Spain 
\&
Instytut Matematyczny,
Uniwersytet Wroc{\l}awski,
Plac Grunwaldzki 2,
50-384 Wroc{\l}aw,
Poland}
\email{tzszarek@bcamath.org}

\thanks{The first author was partially supported by NSF grant DMS-2154712, and by the National Science Centre
in Poland, grant Opus 2018/31/B/ST1/00204. The second author was
partially supported by the National Science Centre in Poland, grant
Opus 2018/31/B/ST1/00204.  The third author was partially supported by
the National Science Centre of Poland, grant Opus 2017/27/B/ST1/01623,
by Juan de la Cierva Incorporaci{\'o}n 2019 grant number IJC2019-039661-I  
funded by
Agencia Estatal de Investigaci{\'o}n, grant PID2020-113156GB-I00/AEI/10.13039/501100011033
and also by the
Basque Government through the BERC 2022-2025 program
and by Spanish Ministry of Sciences, Innovation and
Universities: BCAM Severo Ochoa accreditation SEV-2017-0718.}

\begin{abstract}
In this paper we establish uniform oscillation estimates on $L^p(X)$ with
$p\in(1,\infty)$ for the polynomial ergodic  averages. This result contributes to a 
certain problem about uniform oscillation bounds for ergodic averages  formulated by Rosenblatt and Wierdl in the early 1990's. We
also give a slightly different proof of the uniform oscillation inequality
of Jones, Kaufman, Rosenblatt and Wierdl for bounded
martingales. Finally, we show that oscillations, in contrast to jump
inequalities, cannot be seen as an endpoint for $r$-variation
inequalities.
\end{abstract}


\maketitle
\section{Introduction}
\subsection{Statement of the main results}
For $d, k\in\ZZ_+$ let us consider a polynomial mapping
\begin{equation}\label{polymap}
   P:=(P_1,\dots,P_{d})\colon\ZZ^k\to\ZZ^{d},
\end{equation}
where each $P_j\colon\ZZ^k\to\ZZ$ is a $k$-variate polynomial  with integer coefficients such that $P_j(0)=0$.

Let $\Omega$ be a non-empty convex body (not necessarily symmetric) in $\RR^k$, which simply means that $\Omega$ is a bounded convex open subset of $\RR^k$.
For $t>0$ we define its dilates
\begin{align}
\label{eq:60}
\Omega_t := \{x\in\RR^{k}: t^{-1}x\in\Omega\}.
\end{align}
We will additionally assume that $B(0,c_{\Omega}) \subseteq \Omega \subseteq B(0,1) \subset \RR^{k}$ for some $c_{\Omega}\in(0, 1)$, where $B(x, t)$ denotes an open Euclidean ball in $\RR^k$ centered at $x\in\RR^k$ with radius $t>0$.
This ensures that $\Omega_t\cap\ZZ^k=\{0\}$ for all $t\in(0, 1)$. A typical choice of $\Omega_{t}$ is a ball of radius $t$ associated with some norm on $\RR^{k}$.

For $t\in\XX:=[1, \infty)$, $x\in X$ and $f\in L^0(X)$ (see Section \ref{sec:notation} for appropriate definitions) we can define the corresponding ergodic polynomial averaging operator by
\begin{align}
\label{eq:59}
A_t^Pf(x):
= \frac{1}{\abs{\Omega_t\cap\ZZ^k}}
\sum_{m\in \Omega_t\cap\ZZ^k}
f\big(T_1^{P_1(m)}\cdots T_d^{P_d(m)}   x\big).
\end{align}

The first main result of this note is the following uniform oscillation ergodic theorem.

\begin{theorem}
\label{thm:main}
Let $d, k\in\ZZ_+$, a  polynomial mapping $P$ as in \eqref{polymap} and a set $\Omega_t$ as in \eqref{eq:60} be given. Let $(X, \calB(X), \mu)$ be a
$\sigma$-finite measure space endowed with a family
 of commuting  invertible measure
preserving transformations $T_1,\ldots, T_d: X\to X$. 
Then for every $p\in(1, \infty)$ there exists a constant
$C_{d, k, p, \deg P}>0$ such that for every $f\in L^p(X)$ we have
\begin{align}
\label{eq:64}
\sup_{J\in\ZZ_+}\sup_{I\in \mathfrak S_J(\XX)}
\norm[\big]{O_{I, J}^2(A_{t}^Pf:t\in\XX)}_{L^p(X)}
\le C_{d, k, p, \deg P}\norm{f}_{L^p(X)};
\end{align}
we refer to \eqref{eq:13} for the definition of oscillations.
Moreover, the implied constant in \eqref{eq:64} is independent of the
coefficients of the polynomial mapping $P$.
\end{theorem}

We now give some remarks about Theorem \ref{thm:main}.

\begin{enumerate}[label*={\arabic*}.]

\item Theorem \ref{thm:main} is a contribution to a problem from the
early 1990's of Rosenblatt and Wierdl \cite[Problem 4.12, p. 80]{RW}
about uniform estimates of oscillation inequalities for ergodic
averages. This problem has a long and interesting history, which we
briefly describe below.

\item Inequality \eqref{eq:64} is a useful tool, as it was shown by
Bourgain \cite{B1, B2, B3}, in establishing pointwise convergence
for operators \eqref{eq:59}. Inequality \eqref{eq:64} also  implies, in view of
\eqref{eq:63}, that for all $p\in(1, \infty]$, there exists a constant $C_{d, k, p, \deg P}>0$ (with
$C_{d, k, \infty, \deg P}=1$ for $p=\infty$), such that for every $f\in L^p(X)$ we have
\begin{align}
\label{eq:73}
\norm[\big]{\sup_{t\in\XX}\abs{A_t^Pf}}_{L^p(X)}
\le C_{d, k, p, \deg P}\norm{f}_{L^p(X)}.
\end{align}
The constant in \eqref{eq:73} is also independent of the coefficients
of the polynomial mapping $P$.

\item A non-uniform variant of inequality \eqref{eq:64} for one
dimensional averages \eqref{eq:59} with $d=k=1$ was established by
Bourgain in \cite{B1, B2, B3}. More precisely, Bourgain proved that
for any $\tau>1$, any sequence of integers $I=(I_j:{j\in\NN})$ such
that $I_{j+1}>2I_j$ for all $j\in\NN$, and any $f\in L^2(X)$ one has
\begin{align}
\label{eq:17}
\norm[\big]{O_{I, J}^2(A_{\tau^n}^Pf:n\in\NN)}_{L^2(X)}
\le C_{I, \tau}(J)\norm{f}_{L^2(X)}, \qquad J\in\ZZ_+,
\end{align}
where $C_{I, \tau}(J)$ is a constant depending on $I$ and $\tau$ and
such that $\lim_{J\to\infty} J^{-1/2}C_{I, \tau}(J)=0$.
Interestingly, this non-uniform inequality \eqref{eq:17} suffices to establish
pointwise convergence of the averaging operators from \eqref{eq:59} for any  $f\in L^2(X)$, see \cite{B1, B2, B3}.

\item Not long afterwards, Lacey refined Bourgain's argument
\cite[Theorem 4.23, p. 95]{RW}, and showed that for every $\tau>1$
there is a constant $C_{\tau}>0$ such that for any $f\in L^2(X)$ one
has
\begin{align}
\label{eq:18}
\sup_{J\in\ZZ_+}\sup_{I\in \mathfrak S_J(\LL_{\tau})}\norm[\big]{O_{I, J}^2(A_{t}^Pf:t\in\LL_{\tau})}_{L^2(X)}
\le C_{\tau}\norm{f}_{L^2(X)},
\end{align}
where $\LL_{\tau}:=\{\tau^n:n\in \NN\}$. This was the first uniform
oscillation result in the class of $\tau$-lacunary sequences.  Lacey's observation
naturally  motivated a question about uniform estimates, independent of $\tau>1$, of
oscillation inequalities  in \eqref{eq:18},
which for the Birkhoff averages was explicitly formulated in
\cite[Problem 4.12, p. 80]{RW}.

\item In the groundbreaking paper of Jones, Kaufman, Rosenblatt and
Wierdl \cite{jkrw} the authors established Theorem \ref{thm:main} for
the classical Birkhoff averages with $d=k=1$ and $P_1(n)=n$ giving
affirmative answer to \cite[Problem 4.12, p. 80]{RW}. Here our aim
will be to show that \cite[Problem 4.12, p. 80]{RW} remains true for
Bourgain's polynomial ergodic averages even in the multidimensional setting
as in \eqref{eq:59}.

\item Finally we mention that a non-uniform variant of Theorem
\ref{thm:main} was in fact established in \cite{MST2} (see also
\cite{MSZ3}). Specifically, H{\"o}lder's inequality and inequality \eqref{eq:62} and  $r$-variational estimates for $r>2$ (see definition
\eqref{eq:67}) established in \cite{MST2, MSZ3} yield that for every
$p\in(1, \infty)$ there is a constant $C_p>0$ such that for any $r>2$ and every $f\in L^p(X)$ one has
\begin{align*}
\sup_{I\in \mathfrak S_J(\XX)}
\norm[\big]{O_{I, J}^2(A_{t}^Pf:n\in\XX)}_{L^p(X)}
\le C_p\frac{r}{r-2} J^{\frac{1}{2}-\frac{1}{r}} \norm{f}_{L^p(X)}, \qquad J\in\ZZ_+.
\end{align*}

\item The proof of Theorem \ref{thm:main} will be an elaboration of
methods developed in \cite{MST2, MSZ3}.  The main tools are the
Hardy--Littlewood circle method (major arcs estimates Proposition
\ref{aprox}, lattice points estimates Proposition
\ref{prop:lattice-boundary} and Weyl's inequality from Theorem
\ref{thm:weyl-sums}), the Ionescu--Wainger multiplier theory (Theorem
\ref{thm:IW}, see also \cite{IW}), the Rademacher--Menshov argument (inequality \eqref{eq:69}, see also \cite{MSZ2})
and the sampling principle of Magyar--Stein--Wainger (Proposition
\ref{prop:msw}, see also \cite{MSW}). The details are presented in Section \ref{sec:Radon},
where we closely follow the exposition from \cite{MSZ3}. Another
important ingredient of the proof of Theorem \ref{thm:main} is a
uniform oscillation inequality for martingales. Although this
inequality was originally proved in \cite[Theorem 6.4, p. 930]{jkrw}, a slightly different proof is presented in Section
\ref{sec:notation}, see
Proposition \ref{prop:1} for the details.
\end{enumerate}

The second main theorem of this paper is the following counterexample.
\begin{theorem} \label{thm:counter} Let $1 \le p < \infty$ and
$1 < \rho \le r < \infty$ be fixed. It is \textbf{not true} that the
estimate
\begin{align}
\label{eq:10}
\sup_{\lambda > 0} \| \lambda N_{\lambda}(f(\cdot,t):t\in\NN)^{1/r} \|_{\ell^{p,\infty}(\ZZ)} 
\le
C_{p,\rho,r}
\sup_{I\in\mathfrak S_\infty (\NN)} 
 \| O_{I,\infty}^{\rho}(f(\cdot,t):t\in\NN) \|_{\ell^{p}(\ZZ)}
\end{align}
holds uniformly for every measurable function
$f \colon \ZZ \times \NN \to \RR$.

As a consequence of \eqref{eq:10} the following estimate 
\begin{align}
\label{eq:19}
\| V^r(f(\cdot,t):t\in\NN) \|_{\ell^{p,\infty}(\ZZ)} 
\le
C_{p,\rho,r}
\sup_{I\in\mathfrak S_\infty (\NN)} 
\| O_{I,\infty}^{\rho} (f(\cdot,t):t\in\NN) \|_{\ell^{p}(\ZZ)}
\end{align}
\textbf{cannot hold} uniformly for all measurable functions
$f \colon \ZZ \times \NN \to \RR$. We refer to Section
\ref{sec:notation} for definitions of $\rho$-oscillations
\eqref{eq:13}, $r$-variations \eqref{eq:67} and $\lambda$-jumps
\eqref{eq:71}.
\end{theorem}

Theorem \ref{thm:counter} states, in particular, that
$\rho$-oscillation \eqref{eq:13} inequalities cannot be seen (at least
in a straightforward way) as endpoint estimates for $r$-variations
\eqref{eq:67}. It also shows that $\rho$-oscillations are incomparable
with uniform $\lambda$-jumps \eqref{eq:71}.  Our motivation to study
inequalities \eqref{eq:10} and \eqref{eq:19} from Theorem
\ref{thm:counter} arose from the desire of better understanding
relations between $\rho$-oscillations, $r$-variations and
$\lambda$-jumps.  We now use martingales to illustrate these
relations.

The $r$-variations \eqref{eq:67} for a family of bounded martingales
$\mathfrak f=(\mathfrak f_n:X\to \CC:n\in\ZZ_+)$ were studied by L{\'e}pingle
\cite{Lep} who showed that for all $r\in(2, \infty)$ and
$p\in(1, \infty)$ there is a constant $C_{p, r}>0$ such that the following inequality
\begin{align}
\label{eq:150}
\norm{V^r(\mathfrak f_n: n\in\ZZ_+)}_{L^p(X)}
\le C_{p, r}\sup_{n\in\ZZ_+}\norm{\mathfrak f_n}_{L^p(X)}
\end{align}
holds with sharp ranges of exponents, see also \cite{JG} for a
counterexample at $r=2$.  In \cite{Lep} a weak type $(1, 1)$ variants of the
inequality \eqref{eq:150} were proved as well. Inequality
\eqref{eq:150} is an extension of Doob's maximal inequality for
martingales and gives a quantitative form of the martingale
convergence theorem.  We also refer to \cite{PX,B3,MSZ1} for
generalizations and different proofs of \eqref{eq:150}.

Bourgain rediscovered inequality \eqref{eq:150} in his seminal paper
\cite{B3}, where it was used to address the issue of pointwise
convergence of ergodic-theoretic averages along polynomial
orbits. This initiated systematic studies of $r$-variations in
harmonic analysis and ergodic theory, which resulted in a vast
literature \cite{JR1,jkrw,jsw,JG, MST2, MSZ1, MSZ2, MSZ3}.  In
applications in analysis and ergodic theory only $r > 2$ and $p > 1$
matter, and in fact this is the best what we can expect due to the
L{\'e}pingle inequality.

It is not difficult to see that for any sequence of measurable
functions $(\mathfrak{a}_n: n\in\ZZ_+)\subseteq \CC$ one has
\begin{align}
\label{eq:24}
\sup_{\lambda>0}\|\lambda N_{\lambda}(\mathfrak{a}_n: n\in\ZZ_+)^{1/r}\|_{L^p(X)}
\le \|V^r(\mathfrak{a}_n: n\in\ZZ_+)\|_{L^p(X)},
\end{align}
see \eqref{eq:72}. Therefore \eqref{eq:24} combined with \eqref{eq:150} imply jump inequalities for martingales 
for any $r>2$. But as it was first shown  by Pisier and Xu \cite{PX} on $L^2(X)$ and by Bourgain \cite[inequality (3.5)]{B3}
on $L^p(X)$ with $p\in(1, \infty)$  endpoint estimates for $r=2$ are also true. More precisely, for every $p\in(1, \infty)$
there exists a constant $C_p>0$ such that
\begin{align}
\label{eq:25}
\sup_{\lambda>0}\|\lambda N_{\lambda}(\mathfrak{f}_n: n\in\ZZ_+)^{1/2}\|_{L^p(X)}
\le C_p\sup_{n\in\ZZ_+}\norm{\mathfrak f_n}_{L^p(X)}.
\end{align}

A remarkable feature of Bourgain's \cite{B3} approach was based on the observation that inequality \eqref{eq:24} can be reversed in the sense that for every $p\in[1,\infty]$ and $1\le \rho<r\le\infty$ one has
\begin{align}
\label{eq:26}
\| V^{r}( \mathfrak{a}_n : n \in \ZZ_+ ) \|_{L^{p,\infty} (X)}
\lesssim_{p,\rho,r}
\sup_{\lambda>0} \| \lambda N_{\lambda}(\mathfrak{a}_n: n\in \ZZ_+)^{1/\rho} \|_{L^{p,\infty}(X)}.
\end{align}
In Lemma \ref{lem:CEvarjum} it is  shown that one cannot replace $L^{p,\8} (X)$  with $L^{p} (X)$ in \eqref{eq:26}.
In fact, \eqref{eq:25} and \eqref{eq:26} allowed Bourgain \cite{B3} to
recover L{\'e}pingle's inequality \eqref{eq:150}.  Inequality
\eqref{eq:26} is a very striking result (see also \eqref{estt1}) which
states that a priori uniform $\lambda$-jump estimates corresponding to
some $\rho\in[1, \infty)$ and $p\in[1, \infty]$ imply weak type
$r$-variational estimates for the same range of $p$ and for any
$r\in(\rho, \infty]$. Therefore, uniform $\lambda$-jump estimates can
be thought of as endpoint estimates for $r$-variations, even though
the $r$-variations may be unbounded at the endpoint in question, as we
have seen in the context of L{\'e}pingle inequality \eqref{eq:150}
with $r=2$.

This gives us a fairly complete picture of relations between
$r$-variations and $\lambda$-jumps, which immediately lead to a
question about similar phenomenona between $\rho$-oscillations and
$r$-variations as well as $\lambda$-jumps.  This problem  has been undertaken in Theorem
\ref{thm:counter} and arose from the following two
observations. On the one hand, for any $r\ge1$ and any sequence of measurable
functions $(\mathfrak{a}_n: n\in\ZZ_+)\subseteq \CC$ one has
\begin{align}
\label{eq:27}
\sup_{J\in\ZZ_+}\sup_{I\in\mathfrak{S}_{J}(\ZZ_+)}\|O^{r}_{I, J}( \mathfrak{a}_n : n \in \ZZ_+ ) \|_{L^{p} (X)}\le \| V^{r}( \mathfrak{a}_n : n \in \ZZ_+ ) \|_{L^{p} (X)},
\end{align}
which follows from \eqref{eq:62}. Thus \eqref{eq:150} combined with \eqref{eq:27} gives
bounds of $r$-oscillations for martingales on $L^p(X)$ for all 
 $r\in(2, \infty)$ and
$p\in(1, \infty)$.
On the other hand,  it was shown by Jones, Kaufman, Rosenblatt and Wierdl \cite[Theorem 6.4, p. 930]{jkrw} that for every $p\in(1, \infty)$ there is a constant $C_p>0$ such that
\begin{align}
\label{eq:29}
\sup_{J\in\ZZ_+}\sup_{I\in\mathfrak{S}_{J}(\ZZ_+)}\|O^{2}_{I, J}( \mathfrak{f}_n : n \in \ZZ_+ ) \|_{L^{p} (X)}
\le C_{p}\sup_{n\in\ZZ_+}\norm{\mathfrak f_n}_{L^p(X)}.
\end{align}
A slightly different proof of this inequality is given in Proposition
\ref{prop:1}.

Inequalities \eqref{eq:27} and \eqref{eq:29} exhibit a similar
phenomenon to the one that we have seen above in the case of
$\lambda$-jumps (see \eqref{eq:24} and \eqref{eq:25}), where
$2$-variations for martingales explode on $L^p(X)$, but corresponding
$\lambda$-jumps (see inequality \eqref{eq:25}); and oscillations (see
inequality \eqref{eq:29}) are bounded.

This observation gave rise to a natural question whether
$2$-oscillation can be interpreted as an endpoint for $r$-variations
for any $r>2$ in the sense of inequality \eqref{eq:26}.  Theorem
\ref{thm:counter} provides an answer in the negative. A detailed proof
of Theorem \ref{thm:counter} is given in Section \ref{sec:basic},
where a concept of the sequential jump counting function has been
introduced, see \eqref{eq:15}. The sequential jumps can be thought of
as analogues of classical jumps \eqref{eq:71} adjusted to
$\rho$-oscillations, see for instance \eqref{eq:200} and Lemma
\ref{lem:oscmodjum}.  Theorem \ref{thm:counter} also shows that the
space induced by $\rho$-oscillations is different from the spaces
induced by $r$-variations and $\lambda$ jumps corresponding to the
parameter $\rho\le r$.

Even though Theorem \ref{thm:counter} shows that $\rho$-oscillation
inequalities cannot be seen (at least in a straightforward way
understood in the sense of inequality \eqref{eq:26}) as endpoint
estimates for $r$-variations, it is still natural to ask whether a
priori bounds for $2$-oscillations imply bounds for $r$-variations for
any $r>2$. It is an intriguing question from the point of view of
pointwise convergence problems. If it were true it would reduce
pointwise convergence problems to study $2$-oscillations, which in
certain cases are simpler as they are closer to square functions.

The paper is organized as follows. In Section \ref{sec:notation} we
set notation, and collect some important facts about oscillations,
variations and jumps as well as prove Proposition \ref{prop:1}. In
Section \ref{sec:Radon} we give a proof of Theorem
\ref{thm:main}. Finally in Section \ref{sec:basic} we prove our
counterexamples from Theorem \ref{thm:counter}.
\section{Notation and basic tools}\label{sec:notation}

We now set up notation and terminology that will be used throughout
the paper. We also gather basic  properties of jumps as well as
oscillation and variation semi-norms that will be used later.

\subsection{Basic notation}
We denote $\ZZ_+:=\{1, 2, \ldots\}$ and $\NN:=\{0,1,2,\ldots\}$. For
$d\in\ZZ_+$ the sets $\ZZ^d$, $\RR^d$, $\CC^d$ and
$\TT^d:=\RR^d/\ZZ^d$ have a standard meaning.  
For any $x\in\RR$ we will
use the floor function
\[
\lfloor x \rfloor: = \max\{ n \in \ZZ : n \le x \}.
\]
We denote $\RR_+:=(0, \infty)$ and $\XX:=[1, \infty)$, and
for every $N\in\RR_+$ we set
\[
[N]:=(0, N]\cap\ZZ=\{1, \ldots, 
\lfloor N\rfloor\},
\]
and we will also write
\begin{align*}
\NN_{\le N}:= [0, N]\cap\NN,\ \: \quad &\text{ and } \quad
\NN_{< N}:= [0, N)\cap\NN,\\
\NN_{\ge N}:= [N, \infty)\cap\NN, \quad &\text{ and } \quad
\NN_{> N}:= (N, \infty)\cap\NN.
\end{align*}
For $\tau\in(0, 1)$ and $u\in\ZZ_+$ we define sets
    \begin{align*}
        \DD_{\tau}:=\{2^{n^\tau}:n\in\NN\}\quad\text{ and }\quad 2^{u\ZZ_+}:=\{2^{un}\colon n\in\ZZ_+\}.
    \end{align*}

We use $\ind{A}$ to denote the indicator function of a set $A$. If $S$ 
is a statement we write $\ind{S}$ to denote its indicator, equal to $1$
if $S$ is true and $0$ if $S$ is false. For instance $\ind{A}(x)=\ind{x\in A}$.

For two nonnegative quantities $A, B$ we write $A \lesssim B$ if there
is an absolute constant $C>0$ such that $A\le CB$. However, the
constant $C$ may change from line to line. If
$A \lesssim B\lesssim A$, then we write $A \simeq B$.  We will write
$\lesssim_{\delta}$ or $\simeq_{\delta}$ to indicate that the
implicit constant depends on $\delta$. For two functions $f:X\to \CC$ and
 $g:X\to [0, \infty)$, we write $f = \mathcal{O}(g)$ if
there exists a constant $C>0$ such that $|f(x)| \le C g(x)$ for all
$x\in X$. We will  write $f = \mathcal{O}_{\delta}(g)$ if the implicit
constant depends on $\delta$.

\subsection{Euclidean spaces} The standard inner product, the corresponding Euclidean norm, and the maximum norm on $\RR^d$ are denoted respectively, for any $x=(x_1,\ldots, x_d)$, $\xi=(\xi_1, \ldots, \xi_d)\in\RR^d$, by 
\begin{align*}
x\cdot\xi:=\sum_{k=1}^dx_k\xi_k, \qquad \text{ and } \qquad
\abs{x}:=\abs{x}_2:=\sqrt{\ipr{x}{x}}, \qquad \text{ and } \qquad |x|_{\infty}:=\max_{k\in[d]}|x_k|.
\end{align*}

For any multi-index $\gamma=(\gamma_1,\dots,\gamma_k)\in\N^k$, by
abuse of notation we will write
$|\gamma|:=\gamma_1+\ldots+\gamma_k$. This will never cause confusions
since the multi-indices will be always denoted by Greek letters.

Throughout the paper the $d$-dimensional torus $\TT^d$ is a priori endowed with the   periodic norm
\begin{align}
\label{eq:99}
\norm{\xi}:=\Big(\sum_{k=1}^d \norm{\xi_k}^2\Big)^{1/2}
\qquad \text{for}\qquad
\xi=(\xi_1,\ldots,\xi_d)\in\TT^d,
\end{align}
where $\norm{\xi_k}=\dist(\xi_k, \ZZ)$ for all $\xi_k\in\TT$ and
$k\in[d]$.  Identifying $\TT^d$ with $[-1/2, 1/2)^d$ we
see that the norm $\norm{\:\cdot\:}$ coincides with the Euclidean norm
$\abs{\:\cdot\:}$ restricted to $[-1/2, 1/2)^d$.

\subsection{Function spaces}
In this paper all vector spaces  will be defined over  $\CC$. 
The triple $(X, \mathcal B(X), \mu)$ denotes
 a measure space $X$ with a $\sigma$-algebra $\mathcal B(X)$ and a
$\sigma$-finite measure $\mu$.  The space of all $\mu$-measurable
 functions  $f:X\to\CC$ will be denoted by $L^0(X)$.
The space of all functions in $L^0(X)$ whose modulus is integrable
with $p$-th power is denoted by $L^p(X)$ for $p\in(0, \infty)$,
whereas $L^{\infty}(X)$ denotes the space of all essentially bounded
functions in $L^0(X)$.
These notions can be extended to functions taking values in a finite
dimensional normed vector space $(B, \|\cdot\|_B)$, for instance
\begin{align*}
L^{p}(X;B)
:=\big\{F\in L^0(X;B):\|F\|_{L^{p}(X;B)} \coloneqq \left\|\|F\|_B\right\|_{L^{p}(X)}<\infty\big\},
\end{align*}
where $L^0(X;B)$ denotes the space of measurable functions from $X$ to
$B$ (up to almost everywhere equivalence). If $B$ is
separable, these notions can be extended to infinite-dimensional
$B$. However, in this paper by appealing to  standard approximation arguments, we will always be able to work in
finite-dimensional settings.

For any $p\in[1,\infty]$ we define a weak-$L^p$ space of measurable functions on $X$ by setting
\begin{equation*}
    L^{p,\infty}(X):=\{f:X\to\CC\colon\norm{f}_{L^{p,\infty}(X)}<\infty\},
\end{equation*}
where for any $p\in[1,\infty)$ we have
\begin{align*}
\norm{f}_{L^{p,\infty}(X)}:=\sup_{\lambda>0}\lambda\mu(\{x\in X:|f(x)|>\lambda\})^{1/p},
\qquad\text{\and}\qquad
\norm{f}_{L^{\infty,\infty}(X)}:=\norm{f}_{L^{\infty}(X)}.
\end{align*}

In our case we will mainly have  $X=\RR^d$ or
$X=\TT^d$ equipped with the Lebesgue measure, and   $X=\ZZ^d$ endowed with the
counting measure. If $X$ is endowed with a counting measure we will
abbreviate $L^p(X)$ to $\ell^p(X)$ and $L^p(X; B)$ to $\ell^p(X; B)$ and $L^{p,\infty}(X)$ to $\ell^{p,\infty}(X)$.

If $T : B_1 \to B_2$ is a continuous linear  map between two normed
vector spaces $B_1$ and  $B_2$, we use $\|T\|_{B_1 \to B_2}$ to denote its
operator norm.

\subsection{Fourier transform}  
We will use convention that $\ex(z)=e^{2\pi {\bm i} z}$ for
every $z\in\CC$, where ${\bm i}^2=-1$. Let $\calF_{\RR^d}$ denote the Fourier transform on $\RR^d$ defined for
any $f \in L^1(\RR^d)$ and for any $\xi\in\RR^d$ as
\begin{align*}
\calF_{\RR^d} f(\xi) := \int_{\RR^d} f(x) \ex(x\cdot\xi) {\rm d}x.
\end{align*}
If $f \in \ell^1(\ZZ^d)$ we define the discrete Fourier
transform (Fourier series) $\calF_{\ZZ^d}$, for any $\xi\in \TT^d$, by setting
\begin{align*}
\calF_{\ZZ^d}f(\xi): = \sum_{x \in \ZZ^d} f(x) \ex(x\cdot\xi).
\end{align*}
Sometimes we shall abbreviate $\calF_{\ZZ^d}f$ to $\hat{f}$.

Let $\GG=\RR^d$ or $\GG=\ZZ^d$. It is well known that their corresponding dual groups are $\GG^*=(\RR^d)^*=\RR^d$ or $\GG^*=(\ZZ^d)^*=\TT^d$ respectively.
For any bounded function $\mathfrak m: \GG^*\to\CC$ and a test function $f:\GG\to\CC$ we define the Fourier multiplier operator  by 
\begin{align}
\label{eq:2}
T_{\GG}[\mathfrak m]f(x):=\int_{\GG^*}\ex(-\xi\cdot x)\mathfrak m(\xi)\calF_{\GG}f(\xi){\rm d}\xi, \quad \text{ for } \quad x\in\GG.
\end{align}
One may think that $f:\GG\to\CC$ is a compactly supported function on $\GG$ (and smooth if $\GG=\RR^d$) or any other function for which \eqref{eq:2} makes sense.

\subsection{Oscillation semi-norms}
 Let $\II\subseteq \RR$ be 
such that $\# \II\ge2$.
For every $J\in\ZZ_+\cup\{\infty\}$
define 
\begin{align*}
\mathfrak S_J(\II):
=
\Set[\big]{(t_i:i\in\NN_{\le J})\subseteq \II}{t_{0}<
t_{1}<\ldots <t_{J}},
\end{align*}
where $\NN_{\le \infty}:=\NN$.
In other
words, $\mathfrak S_J(\II)$ is a family of all strictly increasing
sequences  of length $J+1$
taking their values in the index set $\II$.

Let $(\mathfrak a_{t}(x): t\in\II)\subseteq\CC$ be a 
family of measurable functions defined on $X$. For any $r\in[1, \infty)$, $\JJ\subseteq\II$ and a sequence
$I=(I_i : i\in\NN_{\le J}) \in \mathfrak S_J(\II)$ the 
oscillation seminorm is defined by
\begin{align}
\label{eq:13}
O_{I, J}^r(\mathfrak a_{t}(x): t \in \JJ):=
\Big(\sum_{j=0}^{J-1}\sup_{t\in [I_j, I_{j+1})\cap\JJ}
\abs{\mathfrak a_{t}(x) - \mathfrak a_{I_j}(x)}^r\Big)^{1/r}.
\end{align}
There will be no problems with measurability in \eqref{eq:13} since we will always assume that
$\II\ni t\mapsto \mathfrak a_{t}(x)\in\CC$ is continuous for
$\mu$-almost every $x\in X$, or $\JJ$ is countable. We also use the
convention that the supremum taken over the empty set is zero.  
\begin{remark}
\label{rem:1}
Some remarks concerning definition \eqref{eq:13} are in order.
\begin{enumerate}[label*={\arabic*}.]
\item It is not difficult to see that $O_{I, J}^r(\mathfrak a_{t}: t \in \JJ)$ defines a semi-norm.

\item  Let $\II\subseteq \RR$ be an index set such that $\#{\II}\ge2$, and let $\JJ_1, \JJ_2\subseteq \II$ be disjoint. Then for any family $(\mathfrak a_t:t\in\II)\subseteq \CC$,  any $J\in\ZZ_+$ and any $I\in\mathfrak S_J(\II)$ one has
\begin{align}
\label{eq:61}
O_{I, J}^r(\mathfrak a_{t}: t\in\JJ_1\cup\JJ_2)
\le O_{I, J}^r(\mathfrak a_{t}: t\in\JJ_1)
+O_{I, J}^r(\mathfrak a_{t}: t\in\JJ_2).
\end{align}

\item  Let $(\mathfrak{a}_t(x): t\in\RR)\subseteq\CC$ be a family of measurable functions on a $\sigma$-finite measure space $(X,\calB(X),\mu)$. Let $\II\subseteq\RR$ and $\#\II\geq2$, then for every $p\in[1,\infty)$ and $r\in[1,\infty)$ we have
\begin{align}
\label{eq:63}
\norm{\sup_{t\in\II\setminus\sup\II}|\mathfrak{a}_t|}_{L^p(X)}\leq\sup_{t\in\II}\norm{\mathfrak{a}_t}_{L^p(X)}+\sup_{N\in\ZZ_+}\sup_{I\in\mathfrak{S}_N(\II)}\norm[\big]{O_{I,N}^r(\mathfrak{a}_t:t\in\II)}_{L^p(X)}.
\end{align}
The inequality \eqref{eq:63} states that oscillations always dominate maximal functions.

\item Let $(\mathfrak{a}_t(x): t\in\RR)\subseteq\CC$ be a family of measurable functions on a $\sigma$-finite measure space $(X,\calB(X),\mu)$. Suppose  that there are $p\in[1,\infty)$, $r\in[1,\infty)$ and $0<C_{p,r}<\infty$ such that
\begin{align}
\label{eq:65}
\sup_{N\in\ZZ_+}\sup_{I\in\mathfrak{S}_N(\RR_+)}\norm{O_{I,N}^r(\mathfrak{a}_t:t\in\RR_+)}_{L^p(X)}\leq C_{p,r}.
\end{align}
Then the limit $\lim_{t\to\infty}\mathfrak{a}_t(x)$ exists for
$\mu$-almost every $x\in X$. In other words, the condition \eqref{eq:65}
 implies pointwise almost everywhere convergence of $\mathfrak{a}_t(x)$ as $t\to\infty$.
\end{enumerate}

\end{remark}

We recall some notation from \cite[Section 3, p. 165]{Hyt}. Let
$(X,\calB(X),\mu)$ be a $\sigma$-finite measure space and let $\II$ be
a totally ordered set. A sequence of sub-$\sigma$-algebras
$(\mathcal F_t: t\in\II)$ is called \textit{filtration} if it is increasing and
the measure $\mu$ is $\sigma$-finite on each $\mathcal F_t$.
Recall that a \textit{martingale} adapted to a filtration $(\mathcal F_t: t\in\II)$ is a family of functions $\mathfrak f=(\mathfrak f_t:t\in\II)\subseteq L^1(X,\calB(X),\mu)$ such that $\mathfrak f_{s}=\EE[\mathfrak f_t|\mathcal F_s]$ for every $s, t\in\II$ so that $s\le t$, where $\EE[\cdot|\mathcal F]$ denotes the the conditional
expectation with respect to a sub-$\sigma$-algebra $\mathcal F\subseteq \mathcal B(X)$. We say that a martingale 
$\mathfrak f=(\mathfrak f_t:t\in\II)\subseteq L^p(X,\calB(X),\mu)$ is bounded if
\[
\sup_{t\in\II}\|\mathfrak f_t\|_{L^p(X)}\lesssim_p1.
\]

We now establish oscillation inequalities for bounded martingales in $L^p(X,\calB(X),\mu)$.
\begin{proposition}\label{prop:1}
For every $p\in(1, \infty)$ there exists a constant $C_p>0$ such that for every
bounded martingale $\mathfrak f=(\mathfrak f_n:n\in\ZZ)\subseteq L^p(X,\calB(X),\mu)$ corresponding to a filtration $(\mathcal F_n: n\in\ZZ)$ one has
\begin{align}
\label{eq:66}
\sup_{N\in\ZZ_+} \sup_{I\in\mathfrak S_N(\ZZ)}
\norm[\big]{O^2_{I,N}(\mathfrak f_n:n\in\ZZ)}_{L^p(X)}
&\leq C_p
\sup_{n\in\ZZ}\|\mathfrak f_n\|_{L^p(X)}.
\end{align}
\end{proposition}

This proposition was established in  \cite[Theorem 6.4, p. 930]{jkrw}. The authors first established \eqref{eq:66} for $p=2$, then proved weak type $(1,1)$ as well as $L^{\infty}\to {\rm BMO}$ variants of \eqref{eq:66}, and consequently by a simple interpolation  derived \eqref{eq:66} for all $p\in(1, \infty)$.  Here we give a slightly different and simplified  proof based on a weighted Doob's inequality, which avoids  $L^{\infty}\to {\rm BMO}$ estimates and in fact is very much in the spirit of the estimates for $p=2$ from \cite[Theorem 6.1, p. 927]{jkrw}.
\begin{proof}[Proof of Proposition \ref{prop:1}]
We fix $N \in\ZZ_+$ and a sequence $I\in\mathfrak S_N(\ZZ)$.
We first prove \eqref{eq:66} for $p\ge 2$. Since $r=p/2\ge1$ we take
a nonnegative $w\in L^{r'}(X)$ such that $\norm{w}_{L^{r'}(X)}\le 1$ and
\begin{align*}
\norm[\big]{O^2_{I,N}(\mathfrak f_n:n\in\ZZ)}_{L^p(X)}^2
&=\sum_{i=0}^{N-1}\int_X\sup_{I_i\le n<I_{i+1}}\abs{\mathfrak f_n-\mathfrak f_{I_i}}^2w\:\dif\mu.
\end{align*}
To estimate the last sum we will use a weighted version of Doob's
maximal inequality, see \cite[Theorem 3.2.3, p. 175]{Hyt},  which asserts that for every $p\in(1, \infty)$,
every function $f\in L^p(X)$ and a nonnegative measurable weight $w$
and for any $n\in\ZZ$ we have
\begin{align}
\label{eq:184}
\bigg(\int_X\sup_{m\le n}\abs{\mathfrak f_m}^pw\:\dif\mu\bigg)^{1/p}
\le p'\bigg(\int_X\abs{\mathfrak f_n}^p\sup_{m\in\ZZ}|\EE[w|\mathcal F_m]|\dif\mu\bigg)^{1/p}.
\end{align}
We will also use an unweighted Doob's inequality, see \cite[Theorem 3.2.2, p. 175]{Hyt}, which yields that for every $p\in(1, \infty)$ we have
\begin{align}
\label{eq:185}
\norm[\big]{\sup_{m\in\ZZ}\abs{\mathfrak f_m}}_{L^{p}(X)}\le p'\sup_{m\in\ZZ}\norm{\mathfrak f_m}_{L^{p}(X)}.
\end{align}
Finally we will use for every $p\in(1, \infty)$  the following bound
\begin{align*}
\sup_{(\omega_n:n\in\ZZ)\in\{-1, 1\}^{\ZZ}}\norm[\Big]{\sum_{k\in\ZZ}\omega_k(\mathfrak f_k-\mathfrak f_{k-1})}_{L^p(X)}
\lesssim_p\sup_{m\in\ZZ}\norm{\mathfrak f_m}_{L^{p}(X)},
\end{align*}
which by Khintchine's inequality ensures
\begin{align}
\label{eq:190}
\norm[\bigg]{\Big(\sum_{i=0}^{N-1}\abs[\big]{\sum_{k=I_i+1}^{I_{i+1}}(\mathfrak f_k-\mathfrak f_{k-1})}^2\Big)^{1/2}}_{L^p(X)}
\lesssim_p
\sup_{m\in\ZZ}\norm{\mathfrak f_m}_{L^{p}(X)}.
\end{align}
Then we conclude
\begin{align*}
\sum_{i=0}^{N-1}\int_X&\sup_{I_i\le n<I_{i+1}}\abs{\mathfrak f_n-\mathfrak f_{I_i}}^2w\:\dif\mu\\
&=\sum_{i=0}^{N-1}\int_X\sup_{I_i\le n<I_{i+1}}\abs{\EE[(\mathfrak f_{I_{i+1}}-\mathfrak f_{I_i})|\mathcal F_n]}^2w\:\dif\mu\\
&\le 4\sum_{i=0}^{N-1} \int_X \abs{\EE[(\mathfrak f_{I_{i+1}}-\mathfrak f_{I_i})|\mathcal F_{I_{i+1}}]}^2\sup_{n\in\ZZ}\abs{\EE[w|\mathcal F_{n}]}\dif\mu
&& \text{by \eqref{eq:184}}\\
&=4\int_X\sum_{i=0}^{N-1}\abs[\big]{\sum_{k=I_i+1}^{I_{i+1}}(\mathfrak f_k-\mathfrak f_{k-1})}^2
\sup_{n\in\ZZ}\abs{\EE[w|\mathcal F_{n}]}\dif\mu\\
&\le 4\norm[\bigg]{\Big(\sum_{i=0}^{N-1}\abs[\big]{\sum_{k=I_i+1}^{I_{i+1}}(\mathfrak f_k-\mathfrak f_{k-1})}^2\Big)^{1/2}}_{L^p(X)}^2
\norm[\big]{\sup_{n\in\ZZ}\abs{\EE[w|\mathcal F_{n}]}}_{L^{r'}(X)}
&&\text{by H{\"o}lder's inequality}\\
&\lesssim_{p, r}\sup_{m\in\ZZ}\norm{\mathfrak f_m}_{L^{p}(X)}^2\norm{w}_{L^{r'}(X)}
&&\text{by \eqref{eq:190} and \eqref{eq:185}}.
\end{align*}
This proves \eqref{eq:66} for all $p\in[2, \infty)$. To prove \eqref{eq:66} for $p\in(1, 2)$ it suffices to show the corresponding  weak type $(1,1)$ estimate. This follows from \cite[Theorem 6.2, p. 928]{jkrw} we omit the details. 
\end{proof}
\subsection{Variation semi-norms}
 We also recall the definition of $r$-variations. For any $\mathbb I\subseteq \RR$, any family $(\mathfrak a_t: t\in\mathbb I)\subseteq \CC$, and any exponent
$1 \leq r < \infty$, the $r$-variation semi-norm is defined to be
\begin{align}
\label{eq:67}
V^{r}(\mathfrak a_t: t\in\mathbb I):=
\sup_{J\in\ZZ_+} \sup_{\substack{t_{0}<\dotsb<t_{J}\\ t_{j}\in\mathbb I}}
\Big(\sum_{j=0}^{J-1}  |\mathfrak a_{t_{j+1}}-\mathfrak a_{t_{j}}|^{r} \Big)^{1/r},
\end{align}
where the latter supremum is taken over all finite increasing sequences in
$\mathbb I$.

\begin{remark}
\label{rem:2}
Some remarks about definition \eqref{eq:67}  are in order.
\begin{enumerate}[label*={\arabic*}.]
\item Clearly $V^{r}(\mathfrak a_t: t\in\mathbb I)$ defines a semi-norm.
\item The function $r\mapsto V^r(\mathfrak{a}_t: t\in \II)$ is non-increasing. Moreover, if $\II_1\subseteq \II_2$, then
\begin{equation*}
    V^r(\mathfrak{a}_t: t\in \II_1)\le V^r(\mathfrak{a}_t: t\in \II_2).
\end{equation*}

 \item  Let $\II\subseteq \RR$ be  such that $\#{\II}\ge2$. Let  $(\mathfrak a_t:t\in\RR)\subseteq \CC$ be given, and let $r\in[1, \infty)$. If $V^r(\mathfrak{a}_t: t\in \RR_+)<\infty$ then
$\lim_{t\to\infty}\mathfrak{a}_t$ exists. Moreover, for any $t_0\in\II$ one has
\begin{align}
\label{eq:68}
\sup_{t\in\II}|\mathfrak{a}_t|\le |\mathfrak{a}_{t_0}|+V^r(\mathfrak{a}_t: t\in \II).
\end{align}

\item Let $\II\subseteq \RR$ be a countable index set such that $\#{\II}\ge2$. Then 
 for any $r\ge1$, and any family $(\mathfrak a_t:t\in\II)\subseteq \CC$, any  $J\in\ZZ_+\cup\{\infty\}$, any  $I\in\mathfrak S_J(\II)$  one has
 \begin{align}
\label{eq:62}
 O_{I, J}^r(\mathfrak a_{t}: t \in \II)\le V^r(\mathfrak{a}_t: t\in \II)\le 2\Big(\sum_{t\in\II}|\mathfrak a_{t}|^r\Big)^{1/r}.
 \end{align}

\item 
The first inequality in \eqref{eq:62} allows us to
deduce the Rademacher--Menshov inequality for oscillations, which
asserts that for any  $j_0, m\in\NN$ so that $j_0< 2^m$ and any
sequence of complex numbers $(\mathfrak a_k: k\in\NN)$, any $J\in[2^m]$ and any $I\in\mathfrak S_J([j_0, 2^m])$ we have
\begin{align}
\label{eq:69}
\begin{split}
O_{I, J}^2(\mathfrak a_{j}: j_0\leq j \le 2^m)&\le V^{2}( \mathfrak a_j: j_0\leq j \le 2^m)\\
&\leq \sqrt{2}\sum_{i=0}^m\Big(\sum_{j=0}^{2^{m-i}-1}\big|\sum_{\substack{k\in U_{j}^i\\
U_{j}^i\subseteq [j_0, 2^m)}} (\mathfrak a_{k+1}-\mathfrak a_{k}) \big|^2\Big)^{1/2},
\end{split}
\end{align}
where $U_j^i:=[j2^i, (j+1)2^i)$ for any $i, j\in\NN$. The latter inequality in \eqref{eq:69} immediately follows from the proof of
\cite[Lemma 2.5, p. 534]{MSZ2}. The inequality \eqref{eq:69} will be used in Section \ref{sec:Radon}.

\item Let $(\mathfrak{a}_t(x): t\in\II)\subseteq\CC$ be a family of measurable functions on a $\sigma$-finite measure space $(X,\calB(X),\mu)$. Then for any $p\ge1$ and $\tau>0$ we have 
\begin{align}
\label{eq:70}
\begin{gathered}
\sup_{N\in\ZZ_+}\sup_{I\in\mathfrak{S}_N(\II)}\norm{O_{I,N}^2(\mathfrak{a}_t:t\in\II)}_{L^p(X)}\lesssim \sup_{N\in\ZZ_+}\sup_{I\in\mathfrak{S}_N(\DD_\tau)}\norm{O_{I,N}^2(\mathfrak{a}_{t}:t\in \DD_\tau)}_{L^p(X)}\\
+\norm[\Big]{\Big(\sum_{n\in\ZZ} V^2(\mathfrak{a}_{t}:t\in[2^{n^\tau},2^{(n+1)^\tau})\cap\II)^2\Big)^{1/2}}_{L^p(X)}.
\end{gathered}
\end{align}
The inequality \eqref{eq:70} is an analogue of \cite[Lemma 1.3, p. 6716]{jsw} for oscillation semi-norms. 
\end{enumerate}

\end{remark}

\subsection{Jumps}
The $r$-variation is closely related to the $\lambda$-jump counting function. Recall that for any $\lambda>0$ the $\lambda$-jump counting function of a function $f : \I \to \C$ is defined by
\begin{align}
\label{eq:71}
N_{\lambda} f:=N_\lambda(f(t):t\in\II):=\sup \{ J\in\N : \exists_{\substack{t_{0}<\ldots<t_{J}\\ t_{j}\in\I}}  : \min_{0 \le j \le J-1} |f(t_{j+1})-f(t_{j})| \ge \lambda \}.
\end{align}

\begin{remark}
\label{rem:3}
Some remarks about definition \eqref{eq:71}  are in order.
\begin{enumerate}[label*={\arabic*}.]
\item For any $\lambda>0$ and a function $f : \I \to \C$ let us also define the following quantity
\begin{align*}
\qquad \mathcal{N}_{\lambda} f:=\calN_\lambda(f(t):t\in\II):=\sup \{ J\in\N : \exists_{\substack{s_1 < t_{1}  \le \ldots \le s_J < t_{J} \\ s_j , t_{j}\in\I}}  : \min_{1 \le j \le J} |f(t_{j})-f(s_{j})| \ge \lambda \}.  
\end{align*}
Then on has $N_{\lambda} f \le \mathcal{N}_{\lambda} f \le N_{\lambda/2} f$.
\item  Let $(\mathfrak{a}_t(x): t\in\RR)\subseteq\CC$ be a family of measurable functions on a $\sigma$-finite measure space $(X,\calB(X),\mu)$. Let $\II\subseteq\RR$ and $\#\II\geq2$, then for every $p\in[1,\infty]$ and $r\in[1,\infty)$ we have 
\begin{align}
\label{eq:72}
\sup_{\lambda>0}\|\lambda N_{\lambda}(\mathfrak{a}_t: t\in\II)^{1/r}\|_{L^p(X)}\le \|V^r(\mathfrak{a}_t: t\in\II)\|_{L^p(X)},
\end{align}
since for all $\lambda>0$ we have the following pointwise estimate
\[
\lambda N_{\lambda}(\mathfrak{a}_t(x): t\in\II)^{1/r}\le V^r(\mathfrak{a}_t(x): t\in\II).
\]
\item  Let $(X,\mathcal{B}(X),\mu)$ be a $\sigma$-finite measure space and $\I\subseteq\RR$. Fix $p \in [1,\8]$, and $1 \le \rho < r \le \infty$. Then for every measurable function $f : X \times \I \to \C$ we have the estimate
\begin{align} \label{estt1}
\big\| V^{r}\big( f(\cdot, t) : t \in \I \big) \big\|_{L^{p,\8} (X)}
\lesssim_{p,\rho,r}
\sup_{\lambda>0} \big\| \lambda N_{\lambda}(f(\cdot, t): t\in \I)^{1/\rho} \big\|_{L^{p,\8}(X)}.
\end{align}
The inequality \eqref{estt1} can be thought of as an inversion of the inequality \eqref{eq:72}. A detailed proof of \eqref{estt1} can be found in \cite[Lemma 2.3, p. 805]{MSZ1}.
\item For every $p \in (1,\infty)$, and $\rho \in (1,\infty)$ there exists a constant $0<C<\infty$ such that for every measure space $(X,\calB(X),\mu)$, and  $\II\subseteq \RR$, there exists a (subadditive) seminorm $\vertiii{\cdot}$ such that
\begin{align}
\label{eq:55}
C^{-1} \vertiii{f} \leq \sup_{\lambda>0} \big\| \lambda N_{\lambda}(f(\cdot, t): t\in \II)^{1/\rho} \big\|_{L^{p}(X)} \leq C \vertiii{f}
\end{align}
 holds for all measurable functions $f : X \times \I \to \CC$. Inequalities in
\eqref{eq:55} were established in \cite[Corollary 2.2, p. 805]{MSZ1}. This
shows that jumps are very close to seminorms.
\end{enumerate}
\end{remark}

We close this discussion by showing that we cannot replace $L^{p,\8} (X)$  with $L^{p} (X)$ in \eqref{estt1}.
\begin{lemma} \label{lem:CEvarjum}
For a fixed $1 \le p < \infty$ there exists a function $f \colon \ZZ_+ \times \ZZ_+ \to \RR$ such that
\begin{align*}
\| V^r(f(\cdot,n):n\in\ZZ_+) \|_{\ell^p(\ZZ_+)} = \8, \qquad  2 \le r \le \infty,
\end{align*}
and
\begin{align*}
\sup_{\lambda > 0} \| \lambda N_{\lambda}(f(\cdot,n):n\in\ZZ_+)^{1/2} \|_{\ell^p(\ZZ_+)} < \8.
\end{align*}
\end{lemma}

\begin{proof}
Take a sequence $(a_x: x\in \ZZ_+)\subseteq \CC$, to be specified later,  such that $a_1 > a_2 > \ldots > 0$ and such that $a_x \to 0$ as $x \to \infty$ (actually $a_x = x^{-1/p}$ will work). We define a function $f$ as follows
\begin{align*}
f(x,1) = a_x, \qquad f(x,n) = 0, \qquad n \in\NN_{\ge2}, \quad x \in \ZZ_+.
\end{align*}
Then, for every $2 \le r \le \infty$ we have $V^r f(x) = a_x$ and consequently we have
\begin{align*}
\| V^r(f(\cdot,n):n\in\ZZ_+) \|_{\ell^p(\ZZ_+)}^p = \sum_{j\in\ZZ_+} a_j^p, \qquad 2 \le r \le \infty.
\end{align*}
We now compute $N_{\lambda}f$. Observe that 
\begin{align*}
N_{\lambda}f (x) =
\begin{cases}
0,  
&
\lambda > a_x,\\
1,  
& \lambda \le a_x.
\end{cases}
\end{align*}
Using this we see that $N_{\lambda}f (x) = 0$ for all $x \in \ZZ_+$ if $\lambda > a_1$. Otherwise, if $a_j \ge \lambda > a_{j+1}$ for some $j \ge 1$, then 
$N_{\lambda}f (x) = \mathds{1}_{[j]} (x)$ and consequently we get
\begin{align*}
\| (N_{\lambda}f)^{1/2} \|_{\ell^p(\ZZ_+)}^p
=
\begin{cases}
0,  
&
\lambda > a_1,\\
j,  
& a_j \ge \lambda > a_{j+1}, \quad j \in\ZZ_+.
\end{cases}
\end{align*}
Therefore, we obtain
\begin{align*}
\sup_{\lambda > 0} \| \lambda (N_{\lambda}f)^{1/2} \|_{\ell^p(\ZZ_+)}^p
& =
\sup_{\lambda > 0} \lambda^p
\begin{cases}
0,  
&
\lambda > a_1,\\
j,  
& a_j \ge \lambda > a_{j+1}, \quad j \in\ZZ_+.
\end{cases} \\
& = 
\sup_{j \in\ZZ_+} j a_j^p.
\end{align*}
Taking $a_j = j^{-1/p}$ the desired conclusion follows and the proof of Lemma~\ref{lem:CEvarjum} is finished.
\end{proof}
\begin{remark}
Note that the same example can be used to show that $\lambda$-jumps with the parameter $r=2$ may not imply $r$-oscillations with $r\ge2$ in the $\ell^p$ sense. To be more precise, there exists a function $f:\ZZ_+\times\ZZ_+\to\RR$ such that  
\begin{equation*}
\sup_{N\in\ZZ_+}\sup_{I\in\mathfrak{S}_N(\ZZ_+)}\norm{O_{I,N}^r(f(\cdot,n):n\in\ZZ_+)}_{\ell^p(\ZZ_+)}
=\infty,\quad 2\le r\le\infty,
\end{equation*}
and
\begin{equation*}
\sup_{\lambda > 0} \| \lambda N_{\lambda}(f(\cdot,n):n\in\ZZ_+)^{1/2} \|_{\ell^p(\ZZ_+)} < \infty.
\end{equation*}
\end{remark}

\section{Uniform oscillation inequalities: Proof of Theorem \ref{thm:main}}\label{sec:Radon}
In this section we prove the main uniform oscillation inequality from
Theorem \ref{thm:main}. We will closely follow the ideas and notation from
\cite{MSZ3}.  We begin with   standard reductions which allow us to
simplify our arguments.  By a standard lifting argument it suffices to
prove \eqref{eq:64} for canonical polynomial mappings. Fix a nonempty $\Gamma\subset \NN^k\setminus\{0\}$  and recall that
the canonical polynomial mapping is defined by
\begin{equation*}
\RR^k\ni x=(x_1,\ldots,x_k)\mapsto(x)^\Gamma:=(x^\gamma\colon\gamma\in\Gamma)\in\RR^\Gamma,
\end{equation*}
where $x^\gamma:=x_1^{\gamma_1}\cdots x_k^{\gamma_k}$ for
$\gamma\in \Gamma$. Here $\RR^\Gamma$ denotes the space
of tuples of real numbers labeled by multi-indices
$\gamma=(\gamma_1,\dots,\gamma_k)$, so that
$\RR^\Gamma\cong\RR^{|\Gamma|}$, and similarly for
$\ZZ^\Gamma\cong\ZZ^{|\Gamma|}$.

For $t\in\XX$, $x\in X$ and $f\in L^0(X)$ define
the averaging operator
\begin{align}
\label{eq:56}
A_tf(x)
= \frac{1}{\abs{\Omega_t\cap\ZZ^k}}
\sum_{m\in \Omega_t\cap\ZZ^k}
f\Big(\prod_{\gamma\in\Gamma}T_{\gamma}^{m^{\gamma}}   x\Big),
\end{align}
where $T_\gamma:X\to X$, $\gamma \in \Gamma$, 
is a family of
commuting  invertible measure preserving transformations.
Our aim will be to prove the following ergodic theorem along the  canonical polynomial mappings.
\begin{theorem}
\label{thm:main'}
Let $k\in\ZZ_+$ and a finite nonempty set 
$\Gamma \subset \NN^{k} \setminus \{0\}$ be given. Let $(X, \calB(X), \mu)$ be a
$\sigma$-finite measure space endowed with a family
 of commuting  invertible measure
preserving transformations $\{T_{\gamma}: X\to X: \gamma\in\Gamma\}$. 
Then for every $p\in(1, \infty)$ there exists a constant
$C_{\Gamma, k, p}>0$ such that for every $f\in L^p(X)$ we have
\begin{align}
\label{eq:57}
\sup_{J\in\ZZ_+}\sup_{I\in \mathfrak S_J(\XX)}
\norm[\big]{O_{I, J}^2(A_{t}f:t\in\XX)}_{L^p(X)}
\le C_{\Gamma, k, p}\norm{f}_{L^p(X)}.
\end{align}
\end{theorem}

We immediately see that Theorem \ref{thm:main'} is a special case of
Theorem \ref{thm:main}. On the other hand, for every polynomial
mapping $P=(P_1,\ldots, P_d):\ZZ^{k} \rightarrow \ZZ^{d}$ as in
\eqref{polymap} there exists a set of multi-indices
$\Gamma=\NN_{\le d_0}^k\setminus \{0\}$ for some $d_0\in\ZZ_+$ such that for any  $j\in[d]$
each  $P_j$ can be written as
\begin{align*}
P_{j}(m) = \sum_{\gamma\in\Gamma} a_{j,\gamma} m^{\gamma}, \qquad m\in\ZZ^k
\end{align*}
for some coefficients $a_{j,\gamma}\in\ZZ$. Setting
\begin{equation}\label{eq:56'}
    T_{\gamma} = \prod_{j=1}^{d} T_{j}^{a_{j,\gamma}},\quad \gamma\in\Gamma,
\end{equation}
where $T_1, \ldots, T_d:X\to X$ is a family of
commuting  invertible measure preserving transformations, we see
that $A^{P}_t=A_t$, since
\[
\prod_{j=1}^dT_j^{P_j(m)}
=\prod_{j=1}^d\prod_{\gamma\in\Gamma}T_j^{a_{j,\gamma} m^{\gamma}}
=\prod_{\gamma\in\Gamma}T_{\gamma}^{m^{\gamma}}.
\]
This shows that the inequality \eqref{eq:57}  from
Theorem \ref{thm:main'} implies the corresponding inequality in Theorem
\ref{thm:main} and thus it suffices to prove Theorem \ref{thm:main'}.
  For this purpose a further reduction to the
set of integers is in order. By invoking Calder{\'o}n's transference \cite{Cald}
principle the matter is reduced to proving the following theorem.

\begin{theorem}
\label{thm:main''}
Let $k\in\ZZ_+$ and a finite nonempty set 
$\Gamma \subset \NN^{k} \setminus \{0\}$ be given.
Let
\begin{align}
\label{eq:74}
M_t f(x):= \frac{1}{\card{\Omega_{t}\cap\ZZ^{k}}} \sum_{y \in\Omega_{t}\cap\ZZ^{k}} f(x-(y)^\Gamma), \qquad x\in \ZZ^{\Gamma},
\end{align}
be an integer counterpart of $A_t$ defined in \eqref{eq:56}. 
Then for every $p\in(1, \infty)$ there exists a constant
$C_{\Gamma, k, p}>0$ such that for every $f\in L^p(X)$ we have
\begin{align}
\label{eq:57'}
\sup_{J\in\ZZ_+}\sup_{I\in \mathfrak S_J(\XX)}
\norm[\big]{O_{I, J}^2(M_{t}f:t\in\XX)}_{\ell^p(\ZZ^\Gamma)}
\le C_{\Gamma, k, p}\norm{f}_{\ell^p(\ZZ^\Gamma)}.
\end{align}
\end{theorem}

This reduction will allows us to use Fourier methods which are 
not available in the abstract measure spaces setting. The operators from
\eqref{eq:74} are sometimes called discrete averaging Radon operators.
From now on we will closely follow the ideas from \cite{MSZ3} to
establish \eqref{eq:57'}. Its proof will be illustrated in the next
few subsections.

\subsection{Ionescu--Wainger multiplier theorem}\label{sec:IW}
The key tool, which will be used in the proof of Theorem~\ref{thm:main''} is the
Ionescu--Wainger theorem \cite{IW}. We now recall the $d$-dimensional
vector-valued Ionescu--Wainger multiplier theorem from \cite[Section 2]{MSZ3}.

\begin{theorem}\label{thm:IW}
For every $\varrho>0$, there exists a family $(P_{\leq N})_{N\in\ZZ_+}$ of subsets of $\ZZ_+$ satisfying:  
\begin{enumerate}[label*={(\roman*)}]
\item \label{IW1} One has $[N]\subseteq P_{\leq N}\subseteq\big[\max\{N, e^{N^{\varrho}}\}\big]$.
\item \label{IW2}  If $N_1\le N_2$, then $P_{\leq N_1}\subseteq P_{\leq N_2}$.
\item \label{IW3}  If $q \in P_{\leq N}$, then all factors of $q$ also lie in $P_{\leq N}$.
\item \label{IW4}  One has $\lcm{(P_N)}\le 3^N$.
\end{enumerate}

Furthermore, for every $p \in (1,\infty)$, there exists  $0<C_{p, \varrho, d}<\infty$ such that, for every $N\in\ZZ_+$, the following holds.

Let $0<\varepsilon_N \le e^{-N^{2\varrho}}$, and let $\Theta : \RR^{d} \to L(H_0,H_1)$ be a measurable function supported on $\varepsilon_{N}\mathbf Q$, where $\mathbf Q:=[-1/2, 1/2)^d$ is a unit cube, with values in the space $L(H_{0},H_{1})$ of bounded linear operators between separable Hilbert spaces $H_{0}$ and $H_{1}$.
Let $0 \leq \mathbf A_{p} \leq \infty$ denote the smallest constant such that, for every function $f\in L^2(\RR^d;H_0)\cap L^{p}(\RR^d;H_0)$, we have
\begin{align}
\label{eq:75}
\norm[\big]{T_{\RR^d}[\Theta]f}_{L^{p}(\RR^{d};H_1)}
\leq
\mathbf A_{p} \norm{f}_{L^{p}(\RR^{d};H_0)}.
\end{align}
Then the multiplier
\begin{equation}
\label{eq:IW-mult}
\Delta_N(\xi)
:=\sum_{b \in\Sigma_{\leq N}}
\Theta(\xi - b),
\end{equation}
where $\Sigma_{\leq N}$ is $1$-periodic subsets of $\TT^d$ defined by
\begin{align}
\label{eq:42}
\Sigma_{\leq N} := \Big\{ \frac{a}{q}\in\QQ^d\cap\TT^d:  q \in P_{\leq N} \text{ and } {\rm gcd}(a, q)=1\Big\},
\end{align}
satisfies for every $f\in L^p(\ZZ^d;H_0)$ the following inequality
\begin{align}
\label{eq:76}
\norm[\big]{ T_{\ZZ^d}[\Delta_{N}]f}_{\ell^p(\ZZ^{d};H_1)}
\le C_{p,\varrho,d}
(\log N) \mathbf A_{p}
\norm{f}_{\ell^p(\ZZ^{d};H_0)}.
\end{align}
\end{theorem}

An important feature of  Theorem~\ref{thm:IW} is that one can
directly transfer square function estimates from the continuous to the
discrete setting. The constant $\mathbf A_{p}$ in \eqref{eq:75} remains unchanged
 when $\Theta $ is replaced by $\Theta(A\cdot)$ for
any invertible linear transformation $A:\RR^d\to \RR^d$. By property \ref{IW1} we see
\begin{align}
\label{eq:43}
|\Sigma_{\leq N}|\lesssim e^{(d+1)N^\varrho}.
\end{align}

A scalar-valued version of Theorem \ref{thm:IW} was originally
established by Ionescu and Wainger \cite{IW} with the factor
$(\log N)^D$ in place of $\log N$ in \eqref{eq:76}. Their proof was
based on a delicate inductive argument that exploited
super-orthogonality phenomena arising from disjoint supports in the definition of \eqref{eq:IW-mult}. A somewhat different proof with factor
$\log N$ in \eqref{eq:76} was given in \cite{M10}. The latter proof,
instead of induction as in \cite{IW}, used certain recursive
arguments, which clarified the role of the underlying square functions
and strong-orthogonalities that lie behind of the proof of the Ionescu--Wainger multiplier theorem. A much broader
context of the super-orthogonality phenomena were recently
discussed in the survey article of Pierce \cite{Pierce}.  A detailed proof of Theorem \ref{thm:IW} (in the spirit of
\cite{M10}) can be found in \cite[Section 2]{MSZ3}. We also refer to
the recent remarkable paper of Tao \cite{TaoIW}, where the factor $\log N$ was
removed from \eqref{eq:76}.

An important ingredient in the proof of Theorem \ref{thm:IW} is the
sampling principle of Magyar--Stein--Wainger from \cite{MSW}. We recall this principle  as it will  play an essential role in our further discussion.

\begin{proposition}
\label{prop:msw}
Let $d\in\ZZ_+$ be fixed. There exists an absolute constant $C>0$ such that the following holds.
Let $p \in [1,\infty]$ and $q\in\ZZ_+$, and let
$B_1, B_2$ be finite-dimensional Banach spaces.  Let
$\mathfrak m : \RR^d \to L(B_1, B_2)$ be a bounded operator-valued
function supported on $q^{-1}\mathbf Q$ and denote the associated
Fourier multiplier operator over $\RR^d$ by $T_{\RR^d}[\mathfrak m]$.  Let
$\mathfrak m^{q}_{\mathrm{per}}$ be the periodic multiplier
\[
\mathfrak m^{q}_{\mathrm{per}}(\xi) : = \sum_{n\in\ZZ^d} \mathfrak m(\xi-n/q),\qquad \xi\in\TT^d.
\]
Then
\[
\|T_{\ZZ^d}[\mathfrak m^{q}_{\mathrm{per}}]\|_{\ell^{p}(\ZZ^d;B_1)\to \ell^{p}(\ZZ^d;B_2)}\le
C\|T_{\RR^d}[\mathfrak m]\|_{L^{p}(\RR^d;B_1)\to L^{p}(\RR^d;B_2)}.
\]
\end{proposition}
The proof can be found in \cite[Corollary 2.1, p. 196]{MSW}.
We also refer to \cite{MSZ1} for generalization of Proposition \ref{prop:msw} to real interpolation spaces, which in particular covers the case of jump inequalities.

\subsection{Proof of inequality \eqref{eq:57'} from Theorem~\ref{thm:main''}}
Fix $p\in(1,\infty)$ and let $f\in \ell^p(\ZZ^\Gamma)$ be a finitely supported function. We fix $p_0>1$, close to $1$ such that $p\in(p_0,p_0')$. We take $\tau\in(0,1)$ such that
\begin{equation}\label{eq:5}
    \tau<\frac{1}{2}\min\{p_0-1,1\}.
\end{equation}
Now by \eqref{eq:70} one can split \eqref{eq:57'}  into long oscillations and short variations respectively
\begin{align}\label{eq:4}
\begin{split}
\sup_{N\in\ZZ_+}\sup_{I\in\mathfrak{S}_N(\XX)}\norm{O_{I,N}^2(M_tf&: t\in\XX)}_{\ell^p(\ZZ^\Gamma)}
\\&\lesssim\sup_{N\in\ZZ_+}\sup_{I\in\mathfrak{S}_N(\DD_\tau)}\norm[\big]{O_{I,N}^2(M_tf:t\in\DD_\tau)}_{\ell^p(\ZZ^\Gamma)}\\
&\,\,+\norm[\Big]{\Big(\sum_{n=0}^\infty V^2\big(M_{t}f:t\in[2^{n^\tau},2^{(n+1)^\tau})\big)^2\Big)^{1/2}}_{\ell^p(\ZZ^\Gamma)}.
\end{split}
\end{align}

To handle the short variations we may proceed as  in \cite{MSZ3} (see also \cite{zk}) and conclude that
\begin{align}
\label{eq:8}
\MoveEqLeft
\norm[\Big]{ \Big( \sum_{n=0}^\infty V^2\big(M_{2^t}f:t\in[n^\tau,(n+1)^\tau)\big)^2\Big)^{1/2}}_{\ell^p(\ZZ^\Gamma)}
\lesssim 
\Big( \sum_{n=0}^\infty n^{-q(1-\tau)} \Big)^{1/q}\norm{f}_{\ell^p(\ZZ^\Gamma)}
 \lesssim
\norm{f}_{\ell^p(\ZZ^\Gamma)},
\end{align}
since $q(1-\tau)>1$ by \eqref{eq:5}; here $q=\min\{p,2\}$. The first inequality in \eqref{eq:8} follows from a simple observation that  for any finite sequence $t_0<t_1<\dotsc<t_J$ contained in $[n^\tau,(n+1)^\tau)$ one has 
\begin{equation*}
\sum_{j=1}^J\norm[\big]{{(M_{2^{t_j}}-M_{2^{t_{j-1}}})\ind{\{0\}}}}_{\ell^1(\ZZ^\Gamma)}
\lesssim 2^{-kn^\tau}\big|\ZZ^k\cap(\Omega_{2^{(n+1)^\tau}}\setminus\Omega_{2^{n^\tau}})\big|.
\end{equation*}
The latter quantity is controlled by
\begin{equation*}
2^{-kn^\tau}|\ZZ^k\cap(\Omega_{2^{(n+1)^\tau}}\setminus\Omega_{2^{n^\tau}})|
\lesssim n^{\tau-1},
\end{equation*}
which immediately follows by invoking the following proposition.
\begin{proposition}[{\cite[Proposition 4.16, p. 45]{MSZ3}, see also \cite[Lemma A.1, p. 554]{MSZ2}}] \label{prop:lattice-boundary}
Let $\Omega\subset\RR^{k}$ be a bounded and convex set and let
$1\leq s\leq \diam(\Omega)$. Then
\begin{equation}\label{eq:6}
\#\{x \in \ZZ^{k} : \dist(x,\partial\Omega)<s\}\lesssim_k s \diam(\Omega)^{k-1}.
\end{equation}
The implicit constant depends only on the dimension $k$, but not on the convex set $\Omega$. 
\end{proposition}

Now we have to bound long oscillations from \eqref{eq:4}. By Davenport's result \cite{Dav} we know that
\begin{equation*}
    \#(\Omega_{2^{n^\tau}}\cap\ZZ^k)=|\Omega_{2^{n^\tau}}|+\calO(2^{n^\tau(k-1)}).
\end{equation*}
Consequently, one has the following estimate
\begin{equation*}
\norm[\big]{M_{2^{n^\tau}}f-T_{\ZZ^\Gamma}[m_{2^{n^\tau}}]f}_{\ell^p(\ZZ^\Gamma)}\lesssim 2^{-n^\tau}\norm{f}_{\ell^p(\ZZ^\Gamma)},
\end{equation*}
where
\begin{equation*}
m_{t}(\xi):=\frac{1}{|\Omega_{t}|}\sum_{y \in\Omega_{t}\cap\ZZ^{k}}\ex(\xi\cdot (y)^\Gamma),\quad\xi\in\TT^\Gamma,\quad t\in \RR_+.
\end{equation*}
Thus we are reduced to prove
\begin{equation}\label{eq:49}
    \sup_{N\in\ZZ_+}\sup_{I\in\mathfrak S_N(\DD_\tau)}\norm[\big]{O_{I,N}^2(T_{\ZZ^\Gamma}[m_{t}]f:t\in\DD_{\tau})}_{\ell^p(\ZZ^\Gamma)}\lesssim\norm{f}_{\ell^p(\ZZ^\Gamma)}.
\end{equation}

\subsection{Proof of inequality \eqref{eq:49}}
Now, let $\chi\in(0,1/10)$. The proof of \eqref{eq:49} will require several appropriately chosen parameters. We choose $\alpha>0$ such that
\begin{equation*}
    \alpha>\left(\frac{1}{p_0}-\frac{1}{2}\right)\left(\frac{1}{p_0}-\frac{1}{\min\{p,p'\}}\right)^{-1}.
\end{equation*}
Let $u\in\ZZ_+$ be a large natural number to be specified later. We set
\begin{equation}
    \varrho:=\min\left\{\frac{1}{10u},\frac{\delta}{8\alpha}\right\},
\end{equation}
where $\delta>0$ is the exponent from the Gauss sum estimates \eqref{gaussum}. Let $S_0:=\max(2^{u\ZZ_+}\cap [1,n^{\tau u}])$.  We shall, by convenient abuse of notation, write
\begin{equation*}
    \Sigma_{\le n^{\tau u}}:=\Sigma_{\leq S_0}.
\end{equation*}
Next, for dyadic integers $S\in2^{u\ZZ_+}$ we define 
\begin{equation*}
    \Sigma_S:=\begin{cases}
    \Sigma_{\leq S},& {\rm if}\quad S=2^u,\\
    \Sigma_{\leq S}\setminus\Sigma_{\leq S/2^u},& {\rm if}\quad S>2^u.
    \end{cases}
\end{equation*}
Then it is easy to see that
\begin{equation}\label{eq:52}
     \Sigma_{\le n^{\tau u}}=\bigcup_{\substack{S\le S_0,\\S\in 2^{u\ZZ_+}}}\Sigma_S.
\end{equation}
Now we define the Ionescu--Wainger projection multipliers.
For this purpose, we introduce a diagonal matrix $A$ of size $|\Gamma| \times |\Gamma|$  such that
$(A v)_\gamma := \abs{\gamma} v_\gamma$ for any $\gamma \in \Gamma$ and $v\in\RR^\Gamma$, and for any $t > 0$ we also define
corresponding dilations by setting $t^A x=\big(t^{|\gamma|}x_{\gamma}: \gamma\in \Gamma\big)$ for
every $x\in\RR^\Gamma$. Let $\eta:\RR^\Gamma\to[0,1]$ be a smooth function such that
\begin{equation*}
    \eta(x)=\begin{cases}
    1, & |x|\le1/(32|\Gamma|), \\
    0, & |x|\ge1/(16|\Gamma|).\end{cases}
\end{equation*}

For any $n\in\ZZ_+$ we define 
\begin{equation}\label{eq:50}
\Pi_{\le n^\tau,\,n^\tau(A-\chi {\rm Id})}(\xi):=\sum_{a/q\in\Sigma_{\le n^{\tau u}}}\eta\big(2^{n^\tau(A-\chi {\rm Id})}(\xi-a/q)\big),\quad \xi\in\TT^\Gamma,
\end{equation}
as well as
\begin{equation}\label{eq:51}
\Pi_{S,\, n^\tau(A-\chi {\rm Id})}(\xi):=\sum_{a/q\in\Sigma_{S}}\eta\big(2^{n^\tau(A-\chi {\rm Id})}(\xi-a/q)\big),\quad \xi\in\TT^\Gamma,\quad S\in 2^{u\ZZ_+}.
\end{equation}
Note that Theorem~\ref{thm:IW} applies for \eqref{eq:50} and \eqref{eq:51} since 
$2^{-n^\tau(|\gamma|-\chi)}\le e^{-n^{\tau/10}}\le e^{-S_0^{2\varrho}}$ for sufficiently large $n\in\ZZ_+$. Using \eqref{eq:50} one see that
\begin{align}
\nonumber    \sup_{N\in\ZZ_+}\sup_{I\in\mathfrak S_N(\DD_\tau)}&\norm[\big]{O_{I,N}^2(T_{\ZZ^\Gamma}[m_{t}]f:t\in\DD_{\tau})}_{\ell^p(\ZZ^\Gamma)}\\
&\lesssim\sup_{N\in\ZZ_+}\sup_{I\in\mathfrak{S}_N(\NN)}\norm[\big]{O_{I,N}^2(T_{\ZZ^\Gamma}[m_{2^{n^\tau}}(1-\Pi_{\le n^\tau,\,n^\tau(A-\chi {\rm Id})})]f:n\in\NN)}_{\ell^p(\ZZ^\Gamma)}\label{eq:11}\\
    &\qquad+\sup_{N\in\ZZ_+}\sup_{I\in\mathfrak{S}_N(\NN)}\norm[\big]{O_{I,N}^2(T_{\ZZ^\Gamma}[m_{2^{n^\tau}}\Pi_{\le n^\tau,\,n^\tau(A-\chi {\rm Id})}]f:n\in\NN)}_{\ell^p(\ZZ^\Gamma)}\label{eq:12}.
\end{align}
The first and second term in the above inequality corresponds to minor and major arcs, respectively. 
\subsection{Minor arcs: estimates for \eqref{eq:11}}
A key ingredient in estimating \eqref{eq:11} will be a multidimensional Weyl's inequality.

\begin{theorem}[{\cite[Theorem A.1, p. 49]{MSZ3}}]
\label{thm:weyl-sums}
For every $k\in\ZZ_+$ and a finite nonempty set 
$\Gamma \subset \NN^{k} \setminus \{0\}$, there exists $\epsilon>0$ such that, for every polynomial
\[
P(n) = \sum_{\gamma\in\Gamma} \xi_{\gamma} n^{\gamma},
\]
every $N>1$, convex set $\Omega\subseteq B(0,N) \subset \RR^{k}$, function $\phi : \Omega\cap\ZZ^{k}\to\CC$, multi-index $\gamma_{0}\in\Gamma$, and integers $0\leq a<q$ with ${\rm gcd}(a, q) = 1$ and
\begin{align*}
\abs[\Big]{\xi_{\gamma_0} - \frac{a}{q}}
\leq
\frac{1}{q^2},
\end{align*}
we have
\begin{equation}
\label{eq:79}
\abs[\Big]{\sum_{n \in \Omega \cap \ZZ^k} \ex(P(n))\phi(n)}
\lesssim_{\Gamma,k}
N^k \kappa^{-\epsilon} \log (N+1) \norm{\phi}_{L^\infty(\Omega)} + N^{k} \sup_{\abs{x-y}\leq N\kappa^{-\epsilon}} \abs{\phi(x)-\phi(y)},
\end{equation}
where
\[
\kappa := \min \{q,N^{\abs{\gamma_{0}}}/q\}.
\]
The implicit constant in \eqref{eq:79} is independent on the coefficients of  $P$ and the numbers $a$, $q$, and $N$.
\end{theorem}

To estimate \eqref{eq:11} we first observe, using \eqref{eq:62}, that
\begin{align}
\label{eq:78}
\begin{gathered}
\sup_{N\in\ZZ_+}\sup_{I\in\mathfrak{S}_N(\NN)}\norm[\big]{O_{I,N}^2(T_{\ZZ^\Gamma}[m_{2^{n^\tau}}(1-\Pi_{\le n^\tau,\,n^\tau(A-\chi {\rm Id})})]f:n\in\NN)}_{\ell^p(\ZZ^\Gamma)}\\
\lesssim
\sum_{n=0}^\infty\norm{T_{\ZZ^\Gamma}[m_{2^{n^\tau}}(1-\Pi_{\le n^\tau,\,n^\tau(A-\chi {\rm Id})})]f}_{\ell^p(\ZZ^\Gamma)}.
\end{gathered}
\end{align}
Using Theorem \ref{thm:weyl-sums} and  proceeding as in \cite[Lemma 3.29, p. 34]{MSZ3} we may conclude that 
\begin{equation}\label{eq:9}
\norm{T_{\ZZ^\Gamma}[m_{2^{n^\tau}}(1-\Pi_{\le n^\tau,\,n^\tau(A-\chi {\rm Id})})]f}_{\ell^p(\ZZ^\Gamma)}
\lesssim(n+1)^{-2}\|f\|_{\ell^p(\ZZ^\Gamma)},
\end{equation}
provided that $u$ is large.
Combining \eqref{eq:78} with \eqref{eq:9} we obtain the desired bound for \eqref{eq:11}.

\subsection{Major arcs: estimates for \eqref{eq:12}}
Our aim is to prove 
\begin{align}\label{eq:44}
   \sup_{N\in\ZZ_+}\sup_{I\in\mathfrak S_N(\NN)}\norm[\big]{O_{I,N}^2(T_{\ZZ^\Gamma}[m_{2^{n^\tau}}\Pi_{\le n^\tau,\,n^\tau(A-\chi {\rm Id})}]f:n\in\NN)}_{\ell^p(\ZZ^\Gamma)}\lesssim\norm{f}_{\ell^p(\ZZ^\Gamma)}.
\end{align}
For $n\in\NN$ and $\xi\in\TT^{\Gamma}$ define a new multiplier
\begin{equation}\label{eq:38}
    {\bm m}_n(\xi):=\sum_{a/q\in\Sigma_{\le n^{\tau u}}}G(a/q)\Phi_{2^{n^\tau}}(\xi-a/q)\eta(2^{n^\tau(A-\chi {\rm Id})}(\xi-a/q)),
\end{equation}
where $\Phi_{N}$ is a continuous version of the multiplier $m_{N}$ given by
\begin{align*}
\Phi_N(\xi):=\frac{1}{|\Omega_{N}|}\int_{\Omega_{N}}{\ex(\xi\cdot(t)^\Gamma){\rm d}t},
\end{align*}
and $G(a/q)$ is the Gauss sum defined by
\begin{equation*}
G(a/q):=\frac{1}{q^k}\sum_{r\in[q]^k}{\ex((a/q)\cdot(x)^\Gamma)}.
\end{equation*}
It is well known from the multidimensional Van der Corput lemma \cite[Proposition 5, p. 342]{bigs} that
\begin{equation}\label{Phi}
    |\Phi_N(\xi)|\lesssim |N^{A}\xi|_\infty^{-1/|\Gamma|},\qquad\text{and}\qquad|\Phi_N(\xi)-1|\lesssim |N^{A}\xi|_\infty.
\end{equation}
By Theorem \ref{thm:weyl-sums} (see also \cite[Lemma 4.14, p. 44]{MSZ3}) one can easily find   $\delta\in(0, 1)$ such that 
\begin{equation}\label{gaussum}
    |G(a/q)|\lesssim_k q^{-\delta}
\end{equation}
for every $q\in\ZZ_+$ and $a=(a_{\gamma}:\gamma\in \Gamma)\in[q]^{\Gamma}$ such that ${\rm gcd}(a, q)=1$.

We claim that
\begin{align}
\label{eq:7}
    \norm[\big]{T_{\ZZ^\Gamma}[m_{2^{n^\tau}}\Pi_{\le n^\tau,\,n^\tau(A-\chi {\rm Id})}-{\bm m}_n]f}_{\ell^p(\ZZ^\Gamma)}\lesssim 2^{-n^{\tau/2}}\norm{f}_{\ell^p(\ZZ^\Gamma)}.
\end{align}
To establish \eqref{eq:7} one can proceed as in \cite[Lemma 3.38, p. 36]{MSZ3} by appealing to the proposition stated below and the estimate \eqref{eq:43}, Theorem~\ref{thm:IW}  and Proposition~\ref{prop:msw}. 

\begin{proposition}[{\cite[Proposition 4.18, p. 47]{MSZ3}}]\label{aprox}
Let $\Omega\subseteq B(0,N)\subset\RR^k$ be a convex set and  $\mathcal{K}\colon\Omega\to\CC$ be a continuous function. Then for any $q\in\ZZ_+$, $a=(a_{\gamma}:\gamma\in \Gamma)\in[q]^{\Gamma}$ such that ${\rm gcd}(a, q)=1$ and $\xi = a/q + \theta \in \RR^{\Gamma}$ we have
\begin{align*}
    &\big|\sum_{y\in\Omega\cap\ZZ^k}{\ex(\xi\cdot(y)^\Gamma)\mathcal{K}(y)}-G(a/q)\int_{\Omega}{\ex(\theta\cdot(t)^\Gamma)\mathcal{K}(t){\rm d}t}\big|\\
    &\lesssim_k\frac{q}{N}N^k\|\mathcal{K}\|_{L^\infty(\Omega)}+N^k\|\mathcal{K}\|_{L^\infty(\Omega)}\sum_{\gamma\in\Gamma}{(q|\theta_\gamma|N^{|\gamma|-1})^{\varepsilon_\gamma}}+N^k\sup_{x,y\in\Omega\colon|x-y|\le q}|\mathcal{K}(x)-\mathcal{K}(y)|,
\end{align*}
for any sequence $(\varepsilon_\gamma\colon\gamma\in\Gamma)\subseteq[0,1]$. The implicit constant is independent of $a,q,N,\theta$ and the kernel $\mathcal{K}$.
\end{proposition}

Using \eqref{eq:7} and \eqref{eq:62} the proof of the inequality \eqref{eq:44} is reduced to showing
\begin{align}
\label{eq:44'}
   \sup_{N\in\ZZ_+}\sup_{I\in\mathfrak S_N(\NN)}\norm[\big]{O_{I,N}^2(T_{\ZZ^\Gamma}[{\bm m}_{n}]f:n\in\NN)}_{\ell^p(\ZZ^\Gamma)}\lesssim\norm{f}_{\ell^p(\ZZ^\Gamma)}.
\end{align}
\subsection{Major arcs: estimates for \eqref{eq:44'}}
For $n\in\ZZ_+$, $S\in 2^{u\ZZ_+}$ and $\xi\in\TT^{\Gamma}$  we define a new multiplier
\begin{equation}\label{eq:20}
{\bm m}_n^S(\xi):=\sum_{a/q\in\Sigma_S}G(a/q)\Phi_{2^{n^\tau}}(\xi-a/q)\eta\big(2^{n^\tau(A-\chi {\rm Id})}(\xi-a/q)\big).
\end{equation}
Then using \eqref{eq:52} and \eqref{eq:20} we see that to prove \eqref{eq:44'}
it suffices to show that
\begin{align}
\sup_{N\in\ZZ_+}\sup_{I\in\mathfrak{S}_N(\NN_{\ge S^{1/(\tau u)}})}\norm[\big]{O_{I,N}^2(T_{\ZZ^\Gamma}[{\bm m}_{n}^S]f:n\in \NN_{\ge S^{1/(\tau u)}})}_{\ell^p(\ZZ^\Gamma)}\lesssim S^{-4\varrho}\norm{f}_{\ell^p(\ZZ^\Gamma)},\label{eq:21}
\end{align}
since $S^{-4\varrho}$ is summable in $S\in 2^{u\ZZ_+}$.

To establish \eqref{eq:21} we define $\tilde{\eta}(x):=\eta(x/2)$ and  two new multipliers
\begin{align*}
\nu_n^S(\xi)&:=\sum_{a/q\in\Sigma_S}\Phi_{2^{n^\tau}}(\xi-a/q)\eta\big(2^{n^\tau(A-\chi {\rm Id})}(\xi-a/q)\big),\\
\mu_S(\xi)&:=\sum_{a/q\in\Sigma_S}G(a/q)\Tilde{\eta}\big(2^{S^{1/u}(A-\chi {\rm Id})}(\xi-a/q)\big).
\end{align*}
Obviously, we have ${\bm m}_n^S=    \nu_n^S\mu_S$ and we see that the estimate \eqref{eq:21}  will follow if we show that
\begin{align}
\norm[\big]{T_{\ZZ^\Gamma}[\mu_S]f}_{\ell^p(\ZZ^\Gamma)}&\lesssim S^{-7\varrho}\norm{f}_{\ell^p(\ZZ^\Gamma)},\label{eq:22}\\
\sup_{N\in\ZZ_+}\sup_{I\in\mathfrak{S}_N(\NN_{\ge S^{1/(\tau u)}})}\norm[\big]{O_{I,N}^2(T_{\ZZ^\Gamma}[\nu_n^S]f:n\in \NN_{\ge S^{1/(\tau u)}})}_{\ell^p(\ZZ^\Gamma)}
&\lesssim S^{3\varrho}\norm{f}_{\ell^p(\ZZ^\Gamma)}\label{eq:23}.
\end{align}
Using Proposition \ref{aprox}, Theorem \ref{thm:IW} and the Gauss sum estimates \eqref{gaussum} we can easily establish \eqref{eq:22}, we refer to \cite[Lemma 3.47 and estimate (3.49) pp. 38--39]{MSZ3} for more details.  Now we return to \eqref{eq:23}. We define $\kappa_S:=\ceil{S^{2\varrho}}$ and split \eqref{eq:23} into small and large scales respectively
\begin{align*}
\text{LHS of }\eqref{eq:23}\lesssim&\sup_{N\in\ZZ_+}\sup_{I\in\mathfrak S_N(\DD^\tau_{\le S})}\norm[\big]{O_{I,N}^2(T_{\ZZ^\Gamma}[\nu_n^S]f:n\in \DD^\tau_{\le S})}_{\ell^p(\ZZ^\Gamma)}\\
&+\sup_{N\in\ZZ_+}\sup_{I\in\mathfrak S_N(\DD^\tau_{\ge S})}\norm[\big]{O_{I,N}^2(T_{\ZZ^\Gamma}[\nu_n^S]f:n\in \DD^\tau_{\ge S})}_{\ell^p(\ZZ^\Gamma)},
\end{align*}
where $\DD^\tau_{\le S}:=\{n\in\ZZ_+:n\in[S^{1/(\tau u)},2^{\kappa_S+1}]\}$ and  $\DD^\tau_{\ge S}:=\{n\in\ZZ_+:n \ge2^{\kappa_S}\}$. 
\subsection{Major arcs: small scales estimates}
Our aim is to prove 
\begin{equation}\label{eq:14}
\sup_{N\in\ZZ_+}\sup_{I\in\mathfrak S_N(\DD^\tau_{\le S})}\norm[\big]{O_{I,N}^2(T_{\ZZ^\Gamma}[\nu_n^S]f:n\in \DD^\tau_{\le S})}_{\ell^p(\ZZ^\Gamma)}\lesssim \kappa_S\log(S)\norm{f}_{\ell^p(\ZZ^\Gamma)}.
\end{equation}
Using the Rademacher--Menschov inequality from \eqref{eq:69} one has
\begin{equation*}
    \text{ LHS of }\eqref{eq:14}\lesssim\sum_{i=0}^{\kappa_S+1}\norm[\bigg]{\Big(\sum_{j}\big|\sum_{n\in U_j^i}T_{\ZZ^\Gamma}[\nu_{n+1}^S-\nu_n^S]f\big|^2\Big)^{1/2}}_{\ell^p(\ZZ^\Gamma)},
\end{equation*}
where $j\in\ZZ_+$ runs over the set of integers such that $U_j^i=[j2^i,(j+1)2^i)\subseteq [S^{1/(\tau u)},2^{\kappa_S+1}]$.
By Theorem \ref{thm:IW} it suffices to justify that uniformly in $0\le i \le \kappa_S + 1$ we have
\begin{align}
\begin{gathered}
\norm[\bigg]{\Big(\sum_{j}\big|\sum_{n\in U_j^i}T_{\RR^d}\big[\Phi_{2^{(n+1)^\tau}}\eta(2^{(n+1)^\tau(A-\chi {\rm Id})}\cdot)-\Phi_{2^{n^\tau}}\eta(2^{n^\tau(A-\chi {\rm Id})}\cdot)\big]f\big|^2\Big)^{1/2}}_{L^p(\RR^\Gamma)}\\
\lesssim \norm{f}_{L^p(\RR^\Gamma)}.
\end{gathered}
\end{align}
This in turn is a fairly straightforward matter by  appealing to \eqref{Phi} and standard arguments from the Littlewood--Paley theory. We refer to \cite{MSZ2} for more details, see also discussion below \cite[Theorem 4.3, p. 42]{MSZ3}.  

\subsection{Major arcs: large scales estimates}
Finally our aim is to prove
\begin{equation}\label{eq:30}
\sup_{N\in\ZZ_+}\sup_{I\in\mathfrak S_N(\DD^\tau_{\ge S})}\norm[\big]{O_{I,N}^2(T_{\ZZ^\Gamma}[\nu_n^S]f:n\in \DD^\tau_{\ge S})}_{\ell^p(\ZZ^\Gamma)}\lesssim\log(S)\norm{f}_{\ell^p(\ZZ^\Gamma)}.
\end{equation}
Proceeding as in \cite[Section 3.6, pp. 40--41]{MSZ3} and invoking Proposition \ref{prop:msw} the estimate \eqref{eq:30} is reduced to showing for every $p\in(1, \infty)$ and $f\in L^p(\RR^{\Gamma})$ the following  uniform oscillation inequality
\begin{equation}\label{eq:46}
\sup_{N\in\ZZ_+}\sup_{I\in\mathfrak S_N(\RR_+)}\norm[\big]{O_{I,N}^2(T_{\RR^d}[\Phi_t]f:t\in\RR_+)}_{L^p(\RR^d)}\lesssim\norm{f}_{L^p(\RR^d)}.
\end{equation}
The inequality \eqref{eq:46} can be reduced to the martingale setting from
Proposition \ref{prop:1} by invoking square function arguments
\cite[Lemma 3.2, p. 6722]{jsw} and the standard Littlewood--Paley
theory. The details may be found in \cite{MSZ2}. This completes the proof of Theorem
\ref{thm:main''}.

\section{Main counterexample: Proof of Theorem \ref{thm:counter}} \label{sec:basic}
In this section we prove our main counterexample.
We introduce the definition of the sequential
$\lambda$-jump counting function. Let $\II\subseteq\RR$, and for a
given increasing sequence $I=(I_j: j\in\NN)\subseteq\II$ and any
$\lambda>0$ the sequential $\lambda$-jump counting function of a
function $f : \I \to \C$ is defined by
\begin{align}
\label{eq:15}
N_{\lambda,I}f:=N_{\lambda,I}(f(t):t\in\II)
:= 
\# \big\{ k \in\N : \sup_{\substack{I_k\le t< I_{k+1}\\t\in\II}}  |f(t) - f(I_k)| \ge \lambda \big\}.
\end{align} 
Then it is easy to see that $N_{\lambda, I} f \le 
\mathcal{N}_{\lambda, I} f \le N_{\lambda/2,I} f$, where 
\begin{align*}
\calN_{\lambda,I}f:=\mathcal{N}_{\lambda, I}(f(t):t\in\II)
:= 
\# \{ k \in\N : \sup_{\substack{I_k\le s,t<I_{k+1}\\s,t\in\II}}  |f(t) - f(s)| \ge \lambda \}.
\end{align*}

Let $\II\subseteq\RR$, and 
an increasing sequence $I=(I_j: j\in\NN)\subseteq\II$ be given. 
For any function $f : \I \to \C$ and $\lambda>0$ we record the following inequality 
\begin{align}
\label{eq:16}
N_{\lambda,I}f\le \calN_{\lambda-\varepsilon}f, \qquad 0 < \varepsilon < \lambda,
\end{align}
with $\calN_{\lambda}f$ defined in Remark~\ref{rem:3}.
\begin{remark}
Note that the sequential $\lambda$-jump counting function \eqref{eq:15} corresponds to the $r$-oscillations in a similar way as the standard $\lambda$-jump counting function \eqref{eq:71}  corresponds to the $r$-variations. Namely,  for any $\lambda>0$, and  $I\in\mathfrak{S}_\infty(\II)$ and $r\ge1$ one has the following pointwise inequality
\begin{equation}\label{eq:200}
\lambda N_{\lambda,I} (f(t) : t \in\II)^{1/r}  \le O^r_{I,\infty} (f(t) : t \in\II),
\end{equation}
which is a straightforward consequence of \eqref{eq:13} and the definition of $N_{\lambda,I}f$. 
\end{remark}
We have the following counterpart of the inequality \eqref{estt1} but stated in terms of the sequential $\lambda$-jump counting function and the oscillation seminorm.
\begin{lemma} \label{lem:oscmodjum}
Let $(X,\mathcal{B}(X),\mu)$ be a $\sigma$-finite measure space and $\I\subseteq\RR$. Fix $p \in [1,\8]$, and $1 \le \rho < r \le \infty$. Then for every measurable function $f : X \times \I \to \C$ we have the estimate
\begin{align} \label{est6}
\sup_{I\in\mathfrak{S}_\infty(\II)}
\big\| O^{r}_{I,\infty} \big( f(\cdot, t) : t \in \I \big) \big\|_{L^{p,\8} (X)}
\lesssim_{p,\rho,r}\sup_{I\in\mathfrak{S}_\infty(\II)} \sup_{\lambda>0} \big\| \lambda N_{\lambda,I}(f(\cdot, t): t\in \I)^{1/\rho} \big\|_{L^{p,\8}(X)}.
\end{align}
\end{lemma}
\begin{proof}
The proof is a repetition of the arguments from \cite[Lemma 2.3, p. 805]{MSZ1}. We omit the details.
\end{proof}
Now we can state the main result of this section.
\begin{lemma}\label{lem:1}
Let $1 \le p < \infty$ and $1 < \rho \le r < \infty$ be fixed. It is \textbf{not true} that the estimate
\begin{align} \label{est21}
\sup_{\lambda > 0} \| \lambda N_{\lambda}(f(\cdot,t):t\in\NN)^{1/r} \|_{\ell^{p,\infty}(\ZZ)} 
\le
C_{p,\rho,r}
\sup_{I\in\mathfrak S_\infty (\NN)} 
 \| O_{I,\infty}^{\rho}(f(\cdot,t):t\in\NN) \|_{\ell^{p}(\ZZ)}
\end{align}
holds uniformly
for every measurable function $f \colon \ZZ \times \NN \to \RR$.
\end{lemma}
Before we prove Lemma~\ref{lem:1} let us state its consequences.
\begin{corollary} \label{cor:1}
Let $1 \le p < \infty$ and $1 < \rho \le r < \infty$ be fixed. Then the following estimates are \textbf{not true} uniformly
for all measurable function $f \colon \ZZ \times \NN \to \RR$:
\begin{align*}
    \| V^r(f(\cdot,t):t\in\NN) \|_{\ell^{p,\infty}(\ZZ)} 
&\lesssim_{p,\rho,r}
\sup_{I\in\mathfrak S_\infty (\NN)} 
 \| O_{I,\infty}^{\rho} (f(\cdot,t):t\in\NN) \|_{\ell^{p}(\ZZ)},\\
 \| V^r(f(\cdot,t):t\in\NN) \|_{\ell^{p,\infty}(\ZZ)} 
&\lesssim_{p,\rho,r}
\sup_{I\in\mathfrak S_\infty (\NN)}  
\sup_{\lambda > 0} \| \lambda (N_{\lambda,I}(f(\cdot,t):t\in\NN)^{1/\rho} \|_{\ell^{p}(\ZZ)},\\
\sup_{\lambda > 0} \| \lambda (N_{\lambda}(f(\cdot,t):t\in\NN)^{1/r} \|_{\ell^{p,\infty}(\ZZ)} 
&\lesssim_{p,\rho,r}
\sup_{I\in\mathfrak S_\infty (\NN)} 
\sup_{\lambda > 0} \| \lambda (N_{\lambda,I}( f(\cdot,t):t\in\NN)^{1/\rho} \|_{\ell^{p}(\ZZ)}.
\end{align*}
\end{corollary} 

\begin{proof}
The result is a simple consequence of Lemma~\ref{lem:1} and inequalities \eqref{eq:72}, \eqref{eq:200}.
\end{proof}

We will produce a counterexample to \eqref{est21} on a bit different space than $\ZZ$ equipped with the counting measure. However, our counterexample can be easily  transferred to the set of integers  $\ZZ$. 

Let $X:=\{ (2^k,n) : k \in \NN, \, n \in \ZZ_+ \}$ be equipped with the measure $\mu$ given by 
\begin{align*}
\mu(\{(2^k,n)\}) := 2^k, \qquad k \in \NN, \quad n \in \ZZ_+.
\end{align*}
For $j,M \in \ZZ_+$ we define a sequence $f_{j,M} \colon X \times \NN \to [0,1]$ as follows. If $k\in\NN_{\le M}$, $n \in [2^{M-k}]$ and $t \in \NN$ we set
\begin{align*}
&f_{j,M} (2^k,n, t):
 =
\sum_{0 \le m < 2^{j-1} } \big( t-2^k((n-1) 2^{j} +2m) \big) 2^{-k} 
\mathds{1}_{\big[ 2^k((n-1) 2^{j} + 2m),\,\, 2^k((n-1) 2^{j} + 2m+1) \big) } (t) \\
&\qquad +\sum_{0 \le m < 2^{j-1} } \big( -t + 2^k((n-1) 2^{j} + (2m+2)) \big) 2^{-k}\mathds{1}_{\big[2^k((n-1) 2^{j} + (2m+1)),\,\, 2^k((n-1) 2^{j} + (2m+2)) \big) } (t).
\end{align*}
Further, 
we put $f_{j,M} (2^k,n,t) = 0$ if $k\in\NN_{>M}$ or $k\in\NN_{\le M}$ but $n\in \NN_{>2^{M-k}}$.

Let us observe that for every $j,M \in \ZZ_+$ and any sequence $\bar{I}=(I_n:n\in\NN)$ we have
\begin{align} \label{est22}
\support N_{2^{-N}} f_{j,M} 
& \subseteq \big\{ (2^k,n) : k\in\NN_{\le M}, \,  n\in [2^{M-k}] \big\}, \quad N \in \NN, \\ \label{est19}
N_{2^{-N} } f_{j,M} (x) 
& =
N_{2^{-M}} f_{j,M} (x), \qquad N \in \NN_{\ge M}, \quad x \in X, \\ \label{est23}
N_{2^{-N}, \bar{I}} f_{j,M} (x) 
& =
N_{2^{-M}, \bar{I}} f_{j,M} (x), \quad N \in \NN_{\ge M}, \quad x \in X, \\ \label{est24}
N_{2^{-N}} f_{j,M} (2^k,n) 
& =
\begin{cases}
2^{j+k},  
& \quad
k\in\NN_{\le N}, \\
2^{j+N},  
& \quad
k\in \NN_{\le M}\setminus \NN_{< N}, \\
\end{cases}
\qquad N\in\NN_{\le M}, \quad  n\in [2^{M-k}], \\ \label{est25}
N_{2^{-N},\bar{I} } f_{j,M} (x) 
& \le 
N_{2^{-N}} f_{j,M} (x) \le 2^{j+N}, \quad N \in \NN, \quad x \in X.
\end{align}

The next lemma will be a key ingredient in the proof 
of Lemma~\ref{lem:1}, allowing to control the right-hand side of \eqref{est21} from above. The latter estimates after appropriate choice of the parameters $j, M$ will lead us to the desired conclusion in Lemma \ref{lem:1}. 

\begin{lemma} \label{lem:2}
For $j,M \in \ZZ_+$,  $N\in \NN_{\le M}$ and $W\in \NN_{\le j+N}$ one has 
\begin{align} \label{est26}
\begin{split}
&\sup_{\bar{I}\in\mathfrak S_\infty (\NN)} 
\mu\big(\{ x \in X : N_{2^{-N}, \bar{I}} f_{j,M} (x) \ge 2^W
\}\big) \\
&\qquad\qquad  \le 
2^{M+j+2+N-W} + (N+1)2^M
+ (M+1) 2^M\mathds{1}_{\{0\}}(W).
\end{split}
\end{align}
\end{lemma}

\begin{proof}
Observe first that the case of $W = 0$ is trivial because by using \eqref{est25} and \eqref{est22} we see that the left-hand side of \eqref{est26} is controlled by
\begin{align} \label{est31}
\mu \big( \big\{ (2^k,n) : k\in\NN_{\le M}, \,  n\in [2^{M-k}] \big\} \big)
= (M+1) 2^M.
\end{align}
Thus from now on we may assume that $W \ge 1$ and consequently $2^W \ge 2$.

Let us fix $\bar{I}\in\mathfrak S_\infty (\NN)$.
For $k\in\NN_{\le M}$ let us define $A_k := \big\{ (2^k,n) : N_{2^{-N}, \bar{I}} f_{j,M} (2^k,n) \ge 2^W \big\}$ and $a_k := \# A_k$. The number $a_k$ means that the sequence $\bar{I}$ detects at least $2^W$ jumps of height at least $2^{-N}$ on $a_k$ elements in  the set $\{ (2^k,n) : n\in [2^{M-k}] \}$. Note that
\begin{align} \label{est28}
\mu\big( \{x\in X : N_{2^{-N}, \bar{I}} f_{j,M} (x) \ge 2^W
\}\big) 
= \sum_{k\in\NN_{\le M}} 2^k a_k.
\end{align}
Observe that $\support f_{j,M} (2^k,n, \cdot) \subseteq [0,2^{M+j}]$ and consequently without any loss of generality we may assume that $\bar{I}$ has a finite length with the first term equal to $0$ and the last term  equal to $2^{M+j}$. This shows that
\begin{align} \label{est27}
2^{M+j} = \sum_{l} (I_{l+1} - I_{l}).
\end{align}
Let us assume that $k\in \NN_{\le M}\setminus \NN_{< N}$. Observe that for every element from $A_k$ there exist at least $2^W$ jumps of height at least $2^{-N}$, which are detected by the sequence $\bar{I}$. Hence, for each element of $A_k$ we can always find  $2^W - 1\ge1$ distinct jumps. Furthermore, the existence of  a jump implies the existence of two consecutive terms of $\bar{I}$, say $I_l$ and $I_{l+1}$, satisfying $I_{l+1} - I_l \ge 2^{k-N}$. Consequently, we see that for every $k\in \NN_{\le M}\setminus \NN_{< N}$ there exist $(2^W - 1) a_k$ pairs of consecutive terms of $\bar{I}$ whose difference is at least $2^{k-N}$. Setting $b_{M+1} = 0$ and 
$b_k = \max_{k \le m \le M} (2^W - 1) a_m$ for $N \le k \le M$, we see that the sequence $b_k$ is non-increasing and there are $b_k$ pairs of consecutive terms of $\bar{I}$ whose difference is at least $2^{k-N}$. Thus one obtains 
\begin{align*}
\sum_{l} (I_{l+1} - I_l) 
\ge 
\sum_{k=N}^M 2^{k-N} \big( b_k - b_{k+1} \big),
\end{align*}
as the pairs that were counted at levels $ \ge k+1$ are not counted at level $k$.
This, together with \eqref{est27}, implies
\begin{align*}
2^{M+j}
& \ge 
\sum_{k=N}^M 2^{k-N} b_k
-
\sum_{k=N+1}^M 2^{k-N-1} b_k=
b_N + \sum_{k=N+1}^M 2^{k-N-1} b_k\ge 
2^{W-1} \sum_{k=N+1}^M 2^{k-1-N} a_k,
\end{align*}
where in the last inequality we have used the fact that 
$b_k \ge (2^{W} - 1) a_k \ge 2^{W-1} a_k$.
Combining this with \eqref{est28} and the trivial bound $a_k \le 2^{M-k}$ we obtain
\begin{align*} 
\mu\big( \{x\in X : N_{2^{-N}, \bar{I}} f_{j,M} (x) \ge 2^W\}
\big)
\le  \sum_{k\in \NN_{\le M}} 2^k a_k
\le 
2^{M+j+2+N-W} + (N+1)2^M,
\qquad W \ge 1.
\end{align*}
This completes the proof of Lemma~\ref{lem:2}.
\end{proof}
A simple consequence of Lemma~\ref{lem:2} is the following useful estimate, which will be used in the proof of Lemma~\ref{lem:1} only with $q \ge 1$.
\begin{lemma} \label{lem:3}
Let $q \in \RR_+$ be fixed. Then
\begin{align*}
\sup_{\bar{I}\in\mathfrak S_\infty (\NN) }  \int_X \big( N_{2^{-N}, \bar{I}} f_{j,M} (x) \big)^q \, {\rm d}\mu(x)
\lesssim
(M+1)2^M + 
\begin{cases}
(N+1) 2^{M + (j+N)q },  & \quad q > 1, \\
(j + N) 2^{M + j + N }, & \quad q = 1, \\
2^{M + j + N },  & \quad 0 < q < 1, 
\end{cases}
\end{align*}
uniformly in $j,M \in \ZZ_+$ and $N\in\NN_{\le M}$.
\end{lemma}

\begin{proof}
Using \eqref{est25} and then Lemma~\ref{lem:2} we see that
\begin{align*}
& \int_X \big( N_{2^{-N}, \bar{I}} f_{j,M} (x) \big)^q \, {\rm d}\mu(x) \\
& \quad =
\sum_{W\in\NN_{\le j+N}} 
\int_{ \{ x \in X : N_{2^{-N}, \bar{I}} f_{j,M} (x) \in [2^W, 2^{W+1} )
\} } 
\big( N_{2^{-N}, \bar{I}} f_{j,M} (x) \big)^q \, {\rm d}\mu(x) \\
& \quad \lesssim
\sum_{W\in\NN_{\le j+N}}  2^{Wq} 
\mu \big( \{x \in X : N_{2^{-N}, \bar{I}} f_{j,M} (x) \ge 2^W \}\big) \\
& \quad \lesssim
\sum_{W\in\NN_{\le j+N}}  2^{Wq} 
\big( 2^{M+j+N-W} + (N+1)2^M + (M+1)2^M\mathds{1}_{\{0\}}(W)\big) \\
& \quad \simeq
(M+1) 2^M + (N+1)2^{M+ (j + N)q } + 
2^{M+ j + N }
\begin{cases}
2^{(j+N)(q-1)},&\quad q > 1, \\
j+N,&\quad q = 1, \\
1,&\quad 0 < q < 1. \\
\end{cases}
\end{align*}
This gives the desired estimates.
\end{proof}
Now we are able to prove Lemma~\ref{lem:1}.

\begin{proof}[Proof of Lemma~\ref{lem:1}]
At first we deal with the left-hand side of \eqref{est21}. Let us denote
\begin{align*} 
L_{j,M} := 
\sup_{\lambda > 0} \| \lambda (N_{\lambda}f_{j,M})^{1/r} \|_{L^{p,\infty}(X)},
\qquad j,M \in \ZZ_+. 
\end{align*}
By changing the variable $\alpha \mapsto \lambda a^{1/r}$ we obtain
\begin{align*}
L_{j,M} 
& = 
\sup_{\alpha, \lambda > 0} \alpha 
\mu\big(\{ (2^k,n) : \lambda  \big( N_{\lambda} f_{j,M} (2^k,n) \big)^{1/r} \ge \alpha\}
\big)^{1/p} \\
& =
\sup_{a, \lambda > 0} \lambda a^{1/r}
\mu\big(\{ (2^k,n) :  N_{\lambda} f_{j,M} (2^k,n)  \ge a
\}\big)^{1/p} \\
& \simeq
\sup_{N \in \NN} \sup_{a > 0} 2^{-N} a^{1/r}
\mu\big( \{(2^k,n) : N_{2^{-N}} f_{j,M} (2^k,n) \ge a\}
\big)^{1/p}.
\end{align*}
Further, using \eqref{est19} and \eqref{est24} we get
\begin{align} \nonumber
L_{j,M} 
& \simeq
\sup_{N\in\NN_{\le M} } \sup_{a > 0} 2^{-N} a^{1/r}
\mu\big(\{ (2^k,n) : N_{2^{-N}} f_{j,M} (2^k,n) \ge a\}
\big)^{1/p} \\ \nonumber
& =
\sup_{N\in\NN_{\le M} } \sup_{l\in\NN_{ \le N}} 2^{-N + (j+l)/r} 
\mu\big(\{ (2^k,n) : N_{2^{-N}} f_{j,M} (2^k,n) \ge 2^{j+l}\}
\big)^{1/p} \\  \nonumber
& =
\sup_{N\in\NN_{\le M}} \sup_{l\in\NN_{ \le N}} 2^{-N + (j+l)/r} 
\big( 2^M (M+1-l) \big)^{1/p} \\  \label{est30}
& =
2^{ j/r + M/p } (M+1)^{1/p}.
\end{align}

We now focus on the right-hand side of \eqref{est21}.
Let
\begin{align*}
R_{j,M} := 
\sup_{\bar{I}\in\mathfrak S_\infty (\NN)} 
 \| O_{\bar{I},\infty}^{\rho}(f_{j,M}(\cdot,t): t\in\NN) \|_{L^{p}(X)}, 
\qquad j,M \in \ZZ_+.
\end{align*}
Since $|f_{j,M} (x,t) - f_{j,M} (x,s)| \ge 2^{-M}$, $x \in X$, $s,t \in \NN$, provided that $f_{j,M} (x,t) \ne f_{j,M} (x,s)$, we obtain
\begin{align*}
\big( O_{\bar{I},\infty}^{\rho}( f_{j,M}(x,t): t\in\NN) \big)^\rho
& =
\sum_{N \in [M+1]} \sum_{k=0}^\infty \sup_{ I_k\le t< I_{k+1}} 
|f_{j,M} (x,t) - f_{j,M} (x,I_k)|^{\rho} \\
& \qquad \qquad \qquad \times
\mathds{1}_{\big\{ \sup_{ I_k\le t< I_{k+1}} 
|f_{j,M} (x,t) - f_{j,M} (x,I_k)| \in (2^{-N}, 2^{-N + 1} ]\big\} } \\
& \le 
\sum_{N \in [M+1]} 2^{-N\rho + \rho } N_{2^{-N}, \bar{I}} f_{j,M} (x).
\end{align*}
Using \eqref{est23} we see that the terms corresponding to $N=M+1$ and $N=M$ are comparable and therefore we get 
\begin{align} \label{est29}
O_{\bar{I},\infty}^{\rho}(f_{j,M}(x,t):t\in\NN) 
\lesssim
\Big( \sum_{N \in [M]} 2^{-N\rho} N_{2^{-N}, \bar{I}} f_{j,M} (x) \Big)^{1/\rho}, 
\end{align}
uniformly in $j,M \in \ZZ_+$, $x \in X$ and any sequence $\bar{I}\in\mathfrak{S}_{\infty}(\NN)$.

In what follows we distinguish three cases.

\paragraph{\bf Case 1:} $p = \rho$. 
Using \eqref{est29} and then Lemma~\ref{lem:3} (with $q = 1$) we arrive at
\begin{align*} 
(R_{j,M})^p
& \lesssim
\sup_{\bar{I}\in\mathfrak{S}_{\infty}(\NN)} 
\sum_{N \in [M]} 2^{-N\rho} \int_X N_{2^{-N}, \bar{I}} f_{j,M} (x) \, {\rm d}\mu(x) \\
& \lesssim
\sum_{N \in [M]} 2^{-N\rho} \Big( (M+1) 2^M + (j+N) 2^{M+j+N} \Big) \\
& \simeq 
(M+1) 2^M + j 2^{M+j}.
\end{align*}
This implies
\begin{align*} 
R_{j,M}
\lesssim
(M+1)^{1/p} 2^{M/p} + j^{1/p} 2^{(M+j)/p}, \qquad j,M \in \ZZ_+.
\end{align*}
Consequently, if \eqref{est21} were true, then combining the above estimate with \eqref{est30} we would have
\begin{align*} 
2^{j/r + M/p} (M+1)^{1/p}
\lesssim
(M+1)^{1/p} 2^{M/p} + j^{1/p} 2^{(M+j)/p}, \qquad j,M \in \ZZ_+,
\end{align*}
which is equivalent to
\begin{align*} 
2^{j/r} 
\lesssim
1 + j^{1/p} 2^{j/p} (M+1)^{-1/p}, \qquad j,M \in \ZZ_+.
\end{align*}
Taking $M = 2^{100j}$, we see that
\begin{align*} 
2^{j/r} 
\lesssim
1 + j^{1/p} 2^{-99j/p}, \qquad j \in \ZZ_+.
\end{align*}
Letting $j \to \infty$ we get the contradiction. This finishes the proof of Case 1.

\paragraph{\bf Case 2:} $p > \rho$.
Using \eqref{est29}, the triangle inequality for the $L^{p/\rho} (X)$ norm and Lemma~\ref{lem:3} (with $q = p/\rho$) we obtain
\begin{align*} 
R_{j,M}
& \lesssim
\sup_{\bar{I}\in\mathfrak{S}_{\infty}(\NN)} 
\Big\| \sum_{N \in [M]} 2^{-N\rho} N_{2^{-N}, \bar{I}} f_{j,M} 
\Big\|_{ L^{p/\rho} (X) }^{1/\rho} \\
& \le 
\sup_{\bar{I}\in\mathfrak{S}_{\infty}(\NN)}  \Big(
\sum_{N \in [M]} 2^{-N\rho} \| N_{2^{-N}, \bar{I}} f_{j,M} 
\|_{ L^{p/\rho} (X) } \Big)^{1/\rho} \\
& \lesssim
\Big( 
\sum_{N \in [M]} 2^{-N\rho} \Big( (M+1) 2^M + (N+1) 2^{M+(j+N)p/\rho} \Big)^{\rho/p}  \Big)^{1/\rho} \\
& \simeq 
\Big( 
(M+1)^{\rho/p} 2^{M\rho/p} +  2^{M\rho/p + j}  \Big)^{1/\rho} \\
& \simeq 
(M+1)^{1/p} 2^{M/p} +  2^{M/p + j/\rho}.
\end{align*}
Therefore, if \eqref{est21} were true, then the above estimate together with \eqref{est30} would lead us to the estimate
\begin{align*} 
2^{j/r + M/p} (M+1)^{1/p}
\lesssim
(M+1)^{1/p} 2^{M/p} +  2^{M/p + j/\rho}, \qquad j,M \in \ZZ_+,
\end{align*}
which is equivalent to
\begin{align*} 
2^{j/r} 
\lesssim
1 + 2^{j/\rho} (M+1)^{-1/p}, \qquad j,M \in \ZZ_+.
\end{align*}
Taking $M = \lfloor 2^{100jp/\rho} \rfloor$, we get
\begin{align*} 
2^{j/r} 
\lesssim
1 + 2^{-99j/\rho}, \qquad j \in \ZZ_+,
\end{align*}
which leads to the contradiction if we let $j \to \infty$. 
This finishes the proof of Case 2.

\paragraph{\bf Case 3:} $p < \rho$. 
Using \eqref{est29} and then applying the H\"older inequality (here we use \eqref{est22} and \eqref{est31} as well) and
Lemma~\ref{lem:3} (with $q = 1$) we infer that
\begin{align*} 
(R_{j,M})^p
& \lesssim
\sup_{\bar{I}\in\mathfrak{S}_{\infty}(\NN)} \int_X
\Big( \sum_{N \in [M]} 2^{-N\rho} N_{2^{-N}, \bar{I}} f_{j,M} (x) \Big)^{p/\rho} \, {\rm d}\mu(x) \\
& \le 
\sup_{\bar{I}\in\mathfrak{S}_{\infty} (\NN)}
\Big( \int_X \sum_{N \in [M]} 2^{-N\rho} N_{2^{-N}, \bar{I}} f_{j,M} (x)
\, d\mu(x) \Big)^{p/\rho} \big( (M+1) 2^M \big)^{1-p/\rho} \\
& \lesssim
\big( (M+1) 2^M \big)^{1-p/\rho}
\Big( \sum_{N \in [M]} 2^{-N\rho} \Big( (M+1) 2^M + (j+N) 2^{M+j+N} \Big) \Big)^{p/\rho} \\
& \simeq 
(M+1) 2^M + \big( (M+1) 2^M \big)^{1-p/\rho} \big( j 2^{M+j} \big)^{p/\rho}.
\end{align*}
This shows that
\begin{align*} 
R_{j,M}
& \lesssim
(M+1)^{1/p} 2^{M/p} + 2^{M/p + j/\rho} j^{1/\rho}  (M+1)^{1/p - 1/\rho}.
\end{align*}
Consequently, if \eqref{est21} were true, then combining the above estimate with \eqref{est30} we would get
\begin{align*} 
2^{j/r + M/p} (M+1)^{1/p}
\lesssim
(M+1)^{1/p} 2^{M/p} + 2^{M/p + j/\rho} j^{1/\rho}  (M+1)^{1/p - 1/\rho}, 
\qquad j,M \in \ZZ_+,
\end{align*}
which is equivalent to
\begin{align*} 
2^{j/r} 
\lesssim
1 + 2^{j/\rho} j^{1/\rho} (M+1)^{-1/\rho}, \qquad j,M \in \ZZ_+.
\end{align*}
Taking $M = 2^{100j}$, we see that
\begin{align*} 
2^{j/r} 
\lesssim
1 + j^{1/\rho} 2^{-99j/\rho}, \qquad j \in \ZZ_+.
\end{align*}
Letting $j \to \infty$ we get the contradiction. This finishes the proof of Case 3. Consequently, the proof of Lemma~\ref{lem:1} is finished.
\end{proof}

\end{document}